\documentclass[11pt,a4paper]{article}
\usepackage{amsfonts,amsgen,amstext,amsbsy,amsopn,amsfonts,amssymb,amscd}
\usepackage[leqno]{amsmath}
\usepackage[amsmath,amsthm,thmmarks]{ntheorem}
\usepackage{epsf,epsfig}
\usepackage{float}
\usepackage{dsfont}
\usepackage{ebezier,eepic}
\usepackage{color}
\usepackage{tikz}
\usepackage{multirow}
\usepackage{mathrsfs}
\usepackage{graphicx}
\usepackage{subfigure}
\newtheorem{thm}{Theorem}[section]

\newtheorem{prop}[thm]{Proposition}
\newtheorem{prob}[thm]{Problem}
\newtheorem{lem}[thm]{Lemma}
\newtheorem{example}[thm]{Example}
\newtheorem{false statement}{False statement}
\newtheorem{cor}[thm]{Corollary}
\newtheorem{fact}[thm]{Fact}

\theoremstyle{definition}
\newtheorem{defn}[thm]{Definition}
\newtheorem{claim}[thm]{Claim}
\newtheorem{remark}[thm]{Remark}
\newtheorem{conj}[thm]{Conjecture}

\makeatletter \@addtoreset{equation}{section}

\def\hh{\mathcal{H}}
\def\hm{\mathcal{M}}
\def\hl{\mathcal{L}}
\def\hht{\mathcal{T}}
\def\he{\mathcal{E}}
\def\hf{\mathcal{F}}
\def\hg{\mathcal{G}}

\def\ha{\mathcal{A}}
\def\hb{\mathcal{B}}
\def\hs{\mathcal{S}}
\def\hr{\mathcal{R}}
\def\hc{\mathcal{C}}

\def\hp{\mathcal{P}}

\begin{document}
\title{\bf\Large Best possible bounds on the double-diversity of intersecting hypergraphs }
\date{}
\author{Peter Frankl$^1$, Jian Wang$^2$\\[10pt]
$^{1}$R\'{e}nyi Institute, Budapest, Hungary\\[6pt]
$^{2}$Department of Mathematics\\
Taiyuan University of Technology\\
Taiyuan 030024, P. R. China\\[6pt]
E-mail:  $^1$frankl.peter@renyi.hu, $^2$wangjian01@tyut.edu.cn
}

\maketitle

\begin{abstract}
For a family $\hf\subset \binom{[n]}{k}$ and two elements $x,y\in [n]$ define $\hf(\bar{x},\bar{y})=\{F\in \hf\colon x\notin F,\ y\notin F\}$. The {\it double-diversity} $\gamma_2(\hf)$ is defined as the minimum of $|\hf(\bar{x},\bar{y})|$ over all pairs $x,y$. Let $\hl\subset\binom{[7]}{3}$ consist of the seven lines of the Fano plane. For $n\geq 7$, $k\geq 3$ one defines the {\it Fano $k$-graph} $\hf_{\hl}$ as the collection of all $k$-subsets of $[n]$ that contain  at least one line. It is proven that for $n\geq 13k^2$ the Fano $k$-graph is the essentially unique family maximizing the double diversity over all $k$-graphs without a pair of disjoint edges. Some similar, although less exact results are proven for triple and higher diversity as well.

\vspace{6pt}

\vspace{6pt}
\end{abstract}

\section{Introduction}

Let $[n]=\{1,2,\ldots,n\}$ be the standard $n$-element set, $2^{[n]}$ its power set and $\binom{[n]}{k}$ the collection of all its $k$-subsets, $0\leq k\leq n$. A family $\hf\subset 2^{[n]}$ is called {\it intersecting} if $F\cap F'\neq \emptyset$ for all $F,F'\in \hf$. If all $F\in \hf$ contain a fixed element $x$ then $\hf$ is called a {\it star}. In the opposite case $\mathop{\cap}\limits_{F\in \hf}=\emptyset$ and $\hf$ is called {\it non-trivial}.

For a subset $S\subset [n]$ one defines the {\it link} $\hf(S) = \{F\setminus S\colon S\subset F\in \hf\}$ and $\hf(\overline{S})=\{F\in \hf\colon F\cap S=\emptyset\}$. For $S$, $T\subset [n]$, we also use $\hf(\overline{S},T)=\{F\in \hf\colon F\cap S=\emptyset, T\subset F\}$. For singletons $S=\{x\}$, we use the shorthand $\hf(x)$ and $\hf(\overline{x})$. Note that
\[
|\hf| = |\hf(x)|+|\hf(\overline{x})|.
\]

Define the {\it maximum $\ell$-degree} of $\hf$ by $\Delta_\ell(\hf)=\max_{S\in \binom{[n]}{\ell}} |\hf(S)|$ and the {\it $\ell$-diversity} by $\gamma_{\ell} (\hf) = \min_{S\in  \binom{[n]}{\ell}} |\hf(\overline{S})|$. For $\ell=1$ we omit the $\ell$ and note
\begin{align}\label{ineq-1-1}
|\hf| =\Delta(\hf)+\gamma(\hf).
\end{align}

Both parameters $\Delta(\hf)$ and $\gamma(\hf)$ have proved useful in investigating the size and structure of intersecting families. (cf. \cite{AH}, \cite{HM}, \cite{F87}, \cite{F87-2},
\cite{LP}, \cite{Ku}, \cite{KZ}, \cite{FKu2021} etc.)

Another important parameter is the {\it covering number} $\tau(\hf)$. A set $T$ is called a {\it transversal} (or {\it cover}) if $F\cap T\neq \emptyset$ for all $F\in \hf$. Let $\hht(\hf)$ be the family of all transversals of $\hf$ with size at most $k$. Define
\[
\tau(\hf) =\min\left\{|T|\colon T \in \hht(\hf) \right\}.
\]

Obviously, if $\hf\subset \binom{[n]}{k}$ is intersecting then $\hf\subset \hht(\hf)$ and $\tau(\hf)\leq k$. In their seminal paper \cite{EL}, Erd\H{o}s and Lov\'{a}sz (among other things) investigated
\[
m(k)=\max\left\{|\hf|\colon \hf\subset \binom{[n]}{k} \text{ is intersecting, } n \text{ is arbitrarily large, }\tau(\hf)=k\right\}.
\]

\vspace{6pt}
{\noindent\bf Erd\H{o}s-Lov\'{a}sz Theorem (\cite{EL}).}
\begin{align}\label{ineq-el}
\lfloor (e-1)k!\rfloor \leq m(k)\leq k^k.
\end{align}

Let us note that $m(2)=3$ is trivial, $m(3)=10$ was proved by \cite{EL}. In \cite{Lovasz} Lov\'{a}sz conjectured that $m(k)$ equals the lower bound of \eqref{ineq-el}. However it was disproved in \cite{FOT95} for $k=4$. The constructions in \cite{FOT96} show that
\begin{align}\label{ineq-1-2}
m(k)\geq \left(1+o(1)\right)\left(\frac{k}{2}\right)^k.
\end{align}

The upper bound part of \eqref{ineq-el} was improved in \cite{Tuza}, \cite{Cherkashin}, \cite{AT}, \cite{F19}, \cite{Zakharov}, however it is still open whether $m(k)<\left((1-\varepsilon)k\right)^k$ for some positive constant $\varepsilon$.

\begin{defn}\label{defn-1}
For $0\leq \ell <k$ define
\[
m_{\ell}(k)=\max_{\hf} \min\{|\hf(\overline{S})|\colon |S|=\ell\}
\]
where the maximum is over all $k$-uniform intersecting families satisfying $\tau(\hf)=k$. Note that $m(k)=m_0(k)$.
\end{defn}

From the definition it should be clear that $m_{\ell}(k)$ is a strictly decreasing function of $\ell$. For general intersecting families let us define the diversity function $g_\ell(n,k)$.

\begin{defn}
For $0\leq \ell <k$ and $n\geq 2k$ define
\[
g_{\ell}(n,k) =\max_{\hf} \min\left\{|\hf(\overline{S})|\colon S\in \binom{[n]}{\ell}\right\}
\]
where the maximum is over all intersecting families $\hf\subset \binom{[n]}{k}$.
\end{defn}

Let us note that the easy example $\binom{[2k-1]}{k}$ shows that $g_\ell(n,k)\geq \binom{2k-\ell-1}{k}>0$. Consequently, while investigating $g_\ell(n,k)$ we may always assume that $\tau(\hf)>\ell$. Let us note also that the Erd\H{o}s-Ko-Rado Theorem (\cite{EKR}) is equivalent to
\begin{align}\label{ineq-ekr}
g_0(n,k)=\binom{n-1}{k-1}.
\end{align}

Let us formulate the important stability result of \cite{HM}.

\vspace{6pt}
{\noindent\bf Hilton-Milner Theorem.}  Suppose that $n\geq 2k\geq 4$, $\hf\subset \binom{[n]}{k}$ is intersecting and $\tau(\hf)\geq 2$. Then
\begin{align}\label{ineq-hm}
|\hf| \leq \binom{n-1}{k-1}-\binom{n-k-1}{k-1}+1.
\end{align}

Let us recall the Hilon-Milner Family $\hm(n,k)=\left\{M\in \binom{[n]}{k}\colon 1\in M, M\cap [2,k+1]\neq \emptyset\right\} \cup\{[2,k+1]\}$ and note $\gamma(\hm(n,k))=1$. On the other hand for the {\it triangle-based} family $\hht(n,k):=\left\{T\in \binom{[n]}{k}\colon |T\cap [3]|\geq 2\right\}$ one has $\gamma(\hht(n,k))=\binom{n-3}{k-2}$.

We should also mention that $\gamma(\hht_0(n,k))=\binom{n-3}{k-2}$ holds also for the subfamily
$\hht_0(n,k):=\left\{T\in \binom{[n]}{k}\colon |T\cap [3]|= 2\right\}$. Consequently, $\gamma(\hf)=\binom{n-3}{k-2}$ for all $\hf$ in between: $\hht_0(n,k)\subset \hf\subset \hht(n,k)$.

Let us note that for all $0<\ell<k$, $\hf\subset \hf'$ implies $\gamma_{\ell}(\hf)\leq \gamma_{\ell}(\hf')$. Therefore while investigating $g_{\ell}(n,k)$ we shall tacitly assume that $\hf$ is {\it saturated}, that is, no further edges can be added without violating the intersecting property.

\begin{example}
Let $k>\ell$ and let $\he=\he(\ell+1,\ell)$ be an intersecting $(\ell+1)$-graph satisfying $\tau(\he(\ell+1,\ell))=\ell+1$, $m_\ell(\ell+1)=\gamma_{\ell}(\he(\ell+1,\ell))$. For convenience assume $\he(\ell+1,\ell)\subset 2^{[n]}$. Define
\[
\hf_{\he} =\left\{F\in \binom{[n]}{k}\colon \exists E\in \he, E\subset F\right\}.
\]
\end{example}

Obviously $\hf_{\he}=\he$ for $k=\ell+1$. For $k\geq \ell+2$, $\hf_{\he}$ is intersecting,
\begin{align}
&|\hf_{\he}|=\left(|\he|-o(1)\right)\binom{n-\ell-1}{k-\ell-1}  \text{ and }\label{ineq-1-3}\\[5pt]
&\gamma_{\ell}(\hf_\he) =\left(m_\ell(\ell+1)-o(1)\right)\binom{n-2\ell-1}{k-\ell-1}.\label{ineq-1-4}
\end{align}
Our main results show that \eqref{ineq-1-3} is asymptotically best possible.

\begin{thm}\label{thm-main0}
Suppose that $\hf\subset \binom{[n]}{k}$ is intersecting, $k>\ell\geq 2$ and $n\geq \frac{(\ell+2)(\ell+1)^\ell -(\ell+1)\ell^{\ell}}{m_{\ell}(\ell+1)}k^2$. Then
\begin{align}
\gamma_\ell(\hf)\leq  m_{\ell}(\ell+1)\binom{n-2\ell-1}{k-\ell-1}+ (\ell+1)\ell^{\ell}k\binom{n-2\ell-2}{k-\ell-2}.
\end{align}
Moreover,
\[
\gamma_\ell(\hf)\leq (m_\ell(\ell+1)-1)\binom{n-2\ell-1}{k-\ell-1}+ (\ell+1)\ell^{\ell}k\binom{n-2\ell-2}{k-\ell-2}
\]
unless $\hht^{(\ell+1)}(\hf)$ is
an $(\ell+1)$-graph with $\gamma_\ell(\hht^{(\ell+1)}(\hf))=m_\ell(\ell+1)$.
\end{thm}

Let us define a family generated by the Fano plane.

\begin{example}
Let $\hl=\{L_1,L_2,\ldots,L_7\}$ be the family of  sets of size 3  corresponding to the seven lines of the Fano plane. Without loss of generality, we may assume that
\begin{align*}
&L_1=(1,2,3),\ L_2=(1,4,5),\ L_3=(1,6,7),\ L_4=(2,4,6), \\[5pt]
&L_5=(2,5,7),\ L_6=(3,5,6),\ L_7=(3,4,7).
\end{align*}
 For $n\geq 7$ and $k\geq 3$, let
\[
\hf_{\hl} =\left\{F\in \binom{[n]}{k}\colon F \supset L_i \mbox{ for some } i\in [7]\right\}.
\]
Note that $|\hl(\overline{P})|=2$ for all $P\in \binom{[7]}{2}$ and $|L_i\cap L_j|=1$ for all $1\leq i<j\leq 7$. It is easy to see that
\[
\gamma_2(\hf_{\hl}) = 2\binom{n-5}{k-3}-\binom{n-7}{k-5}.
\]
\end{example}

For the case $\ell=2$ we succeeded in proving the exact bound.

\begin{thm}\label{main-1}
Suppose that $\hf\subset \binom{[n]}{k}$ is intersecting, $k\geq 3$ and $n\geq 13 k^2$. Then
\begin{align}\label{ineq-thm-1}
\gamma_2(\hf)\leq  2\binom{n-5}{k-3}-\binom{n-7}{k-5}.
\end{align}
Moreover, the equality holds if and only if $\hf$ is isomorphic to $\hf_{\hl}$.
\end{thm}

For $\ell=3$, by proving $m_3(4)=3$ we obtain the following result.

\begin{thm}\label{main-2}
Suppose that $\hf\subset \binom{[n]}{k}$ is intersecting,  $k\geq 4$ and $n\geq 71k^2$. Then
\[
\gamma_3(\hf)\leq  3\binom{n-7}{k-4}+108k\binom{n-8}{k-5}.
\]
\end{thm}

For the case $\ell=1$, using completely different methods we proved the following.

\begin{thm}[\cite{F2020},\cite{FW2022-2}]
Suppose that $\hf\subset \binom{[n]}{k}$ is intersecting, $n\geq 36k$. Then
\begin{align}\label{ineq-diversity}
\gamma(\hf)\leq \binom{n-3}{k-2}.
\end{align}
\end{thm}

We should mention that disproving an earlier conjecture of \cite{F17}, Huang \cite{Huang} and Kupavskii \cite{Ku} proved that \eqref{ineq-diversity} does not hold for $n<(2+\sqrt{3})k$.

\section{The proof of Theorem \ref{thm-main0}}

For $\hf\subset\binom{[n]}{k}$, recall that $\hht(\hf)$ is the family of transversals of $\hf$ with size at most $k$. Define the {\it basis} $\hb(\hf)$ as the family of minimal (for containment) sets in $\hht(\hf)$.

\begin{lem}[\cite{F17, FKK2022}]\label{lem-00}
Suppose that $\hf\subset \binom{[n]}{k}$ is a saturated intersecting family. Then (i) and (ii) hold.
\begin{itemize}
  \item[(i)] $\hb(\hf)$ is an intersecting antichain,
  \item[(ii)] $\hf=\left\{H\in \binom{[n]}{k}\colon \exists B\in \hb, B\subset H\right\}$.
\end{itemize}
\end{lem}

For a family $\hb\subset 2^{[n]}$ and an integer $r$, let $\hb^{(r)}$ denote the subfamily consisting of all members of size $r$. Set also $\hb^{(\leq r)}=\cup_{i\leq r}\hb^{(i)}$.

 We prove the following lemma  by a branching process. The same method was also used in  \cite{F17,FKK2022,FW2022}.

\begin{lem}\label{lem-0}
 Let $n>k>\ell$ be positive integers. Suppose that $\hf\subset \binom{[n]}{k}$ is a saturated intersecting family, $\hb=\hb(\hf)$, $t=\tau(\hf)\geq \ell+1$. Assume that $r$ is the smallest integer such that $\tau(\hb^{(\leq r)})\geq \ell+1$. Then there exists $U=\{u_1,u_2,\ldots,u_\ell\}\subset [n]$ such that either (i) or (ii) holds.
 \begin{itemize}
   \item[(i)]  $r>t$ and $U$ is a transversal of $\hb^{(\leq r-1)}$.
   \item[(ii)]$r=t$ and $|\hb^{(t)}(\overline{U})|=\gamma_{\ell}(\hb^{(t)})$.
 \end{itemize}
  Moreover,
\begin{align}\label{ineq-1}
\sum_{r\leq j \leq k} \frac{|\hb^{(j)}(\overline{U})|}{k^{j-\ell-1}}\leq(t-1)r(r-1)^{\ell-1}.
\end{align}
\end{lem}

\begin{proof}
For the proof we use a branching process. During the proof {\it a sequence} $S=(x_1,x_2,\ldots,x_j)$ is an ordered sequence of distinct elements of $[n]$ and we use $\widehat{S}$ to denote the underlying unordered set $\{x_1,x_2,\ldots,x_j\}$. Note that the minimality of $r$ implies  $\tau(\hb^{(\leq r-1)})\leq \ell$. If $r\geq t+1$, then choose a transversal $U=\{u_1,u_2,\ldots,u_\ell\}$ of $\hb^{(\leq r-1)}$. Moreover, if there are many such transversals we choose one so that $|\hb^{(r)}(\overline{U})|$ is minimal. If $r=t$ then choose $U=\{u_1,u_2,\ldots,u_\ell\}$ such that $|\hb^{(r)}(\overline{U})|$ is minimal.

At the beginning, we assign weight 1 to the empty sequence $S_{\emptyset}$.
At the first stage, choose $B_1\in \hb^{(t)}$ containing $u_1$ and define $|B_1\setminus U|$ sequences $(x_1)$ with $x_1\in B_1\setminus U$ and assign the weight $\frac{1}{|B_1\setminus U|}$ to each of them. At the $i$th stage for $i=2,3,\ldots,\ell$, for each sequence $S=(x_1,x_2,\ldots,x_{i-1})$ we may choose $B\in \hb^{(\leq r)}$ such that $\widehat{S}\cap B= \emptyset$ and $B\cap U\neq \emptyset$. Indeed, otherwise every $B\in \mathop{\cup}\limits_{1\leq p\leq \ell}\hb^{(\leq r)}(u_p)$ intersects $\widehat{S}=\{x_1,x_2,\ldots,x_{i-1}\}$. It follows that $\hb^{(\leq r)}(\overline{\{x_1,x_2,\ldots,x_{i-1}\}})\subset \hb^{(\leq r)}(\overline{U})$. Then we may choose an extra vertex $v$ such that $|\hb^{(r)}(\overline{\{x_1,x_2,\ldots,x_{i-1},v\}})|<|\hb^{(r)}(\overline{U})|$, contradicting the minimal choice of $U$.  Thus there is $B\in \hb^{(\leq r)}$ satisfying $\widehat{S}\cap B= \emptyset$ and $B\cap U\neq \emptyset$.  Then we replace $S=(x_1,x_2,\ldots,x_{i-1})$ by $|B\setminus U|$ $i$-sequences of the form $(x_1,x_2,\ldots,x_{i-1},y)$ with $y\in B\setminus U$ and weight $\frac{w(S)}{|B\setminus U|}$.

At the $(\ell+1)$th stage, since $\tau(\hb^{(\leq r)})\geq \ell+1$, for each sequence $S=(x_1,x_2,\ldots,x_{\ell})$ we may choose $B\in \hb^{(\leq r)}$ such that $\widehat{S}\cap B= \emptyset$. Then we replace $S=(x_1,x_2,\ldots,x_\ell)$ by $|B\setminus U|$ $(\ell+1)$-sequences of the form $(x_1,x_2,\ldots,x_\ell,y)$ with $y\in B\setminus U$ and weight $\frac{w(S)}{|B\setminus U|}$.

In each subsequent stage, we pick a sequence $S=(x_1,\ldots,x_p)$ and denote its weight by $w(S)$. If $\widehat{S}\cap B\neq \emptyset$ for all $B\in \hb$ then we do nothing. If $p=k$ and $\widehat{S}\notin \hf$, then we discard $S$ and this will only decrease the total weight. Otherwise we pick $B\in \hb$ satisfying $\widehat{S}\cap B= \emptyset$ and replace $S$ by the $|B\setminus U|$ sequences $(x_1,\ldots,x_p,y)$ with $y\in B\setminus U$ and assign weight $\frac{w(S)}{|B\setminus U|}$ to each of them. Clearly, the total weight is always at most 1.

We continue until  $\widehat{S}\cap B\neq \emptyset$ for all sequences $S$ and all $B\in \hb$. Note that in each stage if the chosen sequence does not satisfy the stopping rule then it is either replaced by longer sequences or discarded. Moreover, all the sequences have length at most $k$. Thus eventually the process stops. Let $\hs$ be the collection of sequences that survived in the end of the branching process and let $\hs^{(j)}$ be the collection of sequences in $\hs$ with length $j$.

\begin{claim}\label{claim-0}
To each $B\in \hb^{(j)}(\overline{U})$ with $j\geq r$ there is some sequence $S\in \hs^{(j)}$ with $\widehat{S}=B$.
\end{claim}
\begin{proof}
Let us suppose the contrary and let $S=(x_1,\ldots,x_p)$ be a sequence of maximal length that occurred
at some stage of the branching process satisfying $\widehat{S}\subsetneqq B$. Since $\hb$ is intersecting,
 $B\cap B_1\neq \emptyset$. Moreover $B\cap U=\emptyset$ implies $B\cap (B_1\setminus U)\neq \emptyset$ whence $p\geq 1$. Since $\widehat{S}$ is a proper subset of $B$, there exists $F\in \hf$ with $\widehat{S} \cap F=\emptyset$. Then we can find  $B'\in \hb$ with $B'\subset F$ such that $\widehat{S} \cap B'=\emptyset$. Thus at some point we picked $S$ and some $\tilde{B}\in \hb$ with $\widehat{S} \cap \tilde{B}= \emptyset$. Since by Lemma \ref{lem-00} (i) $\hb$ is intersecting, $B\cap \tilde{B}\neq \emptyset$. Moreover, $B\in \hb^{(j)}(\overline{U})$ implies $B\cap (\tilde{B}\setminus U)\neq \emptyset$. Consequently, for each $y\in B\cap (\tilde{B}\setminus U)$ the sequence $(x_1,\ldots,x_p,y)$ occurred in the branching process. This contradicts the maximality of $p$. Hence there is an $S$ at some stage satisfying $\widehat{S}= B$. Since $\hb$ is intersecting, $\widehat{S}\cap B'\neq \emptyset$ for all $B'\in \hb$. Thus $S\in \hs$ and the claim holds.
\end{proof}

By Claim \ref{claim-0}, we see that $|\hb^{(j)}(\overline{U})|\leq |\hs^{(j)}|$ for $j\geq r$.  Let $S=(x_1,\ldots,x_j)\in \hs^{(j)}$ and let $S_i=(x_1,\ldots,x_i)$ for $i=1,\ldots,j$. Note that $u_1\in B_1$ implies $w(S_1)\geq \frac{1}{t-1}$. For $i\geq 2$ assume that $B_i$ is the selected set when replacing $S_{i-1}$ in the  branching process. Clearly,  $x_i\in B_i$ and
\[
w(S)= \prod_{i=1}^j \frac{1}{|B_i\setminus U|}.
\]
For $2\leq i\leq \ell$, since $B_i\in \hb^{(\leq r)}$ and $B_i\cap U\neq \emptyset$ imply $|B_i\setminus  U|\leq r-1$, $w(S_i)\geq \frac{1}{(t-1)(r-1)^{i-1}}$. Since $B_{\ell+1}\in \hb^{(\leq r)}$ implies $|B_i\setminus  U|\leq r$, $w(S_{\ell+1})\geq \frac{1}{r(t-1)(r-1)^{\ell-1}}$. Note that $|B_i\setminus U|\leq k$ for $\ell+2\leq i\leq j$. It follows that
\[
w(S)\geq \frac{1}{r(t-1)(r-1)^{\ell-1}k^{j-\ell-1}}.
\]
Thus,
\[
\sum_{r\leq j \leq k} \frac{|\hb^{(j)}(\overline{U})|}{r(t-1)(r-1)^{\ell-1}k^{j-\ell-1}}\leq \sum_{t\leq j \leq k} \sum_{S\in \hs^{(j)}}w(S) =\sum_{S\in \hs}w(S)\leq 1.
\]
\end{proof}

\begin{proof}[Proof of Theorem \ref{thm-main0}]
 Let $\hf\subset \binom{[n]}{k}$ be an intersecting family. We may assume that $t=\tau(\hf)\geq \ell+1$ and $\hf$ is saturated. Let $\hb=\hb(\hf)$ and let $r$ be the smallest integer such that $\tau(\hb^{(\leq r)})\geq \ell+1$. Clearly $r\geq t\geq \ell+1$.

 Let $U$ be defined in  Lemma  \ref{lem-0} and let us distinguish two cases.

\vspace{5pt}
{\bf Case 1.} $r\geq \ell+2$.

Note that
\begin{align*}
\gamma_\ell(\hf) \leq |\hf(\overline{U})|\leq \sum_{r\leq j\leq k} |\hb^{(j)}(\overline{U})|\binom{n-j-\ell}{k-j}&\leq \sum_{r\leq j\leq k} \frac{|\hb^{(j)}(\overline{U})|}{k^{j-\ell-1}} k^{j-\ell-1}\binom{n-j-\ell}{k-j}.
\end{align*}
Since for $r\leq j\leq k-1$ and $n\geq k^2$,
\begin{align}\label{ineq-2-2}
\frac{k^{j-\ell-1}\binom{n-j-\ell}{k-j}}{k^{j-\ell}\binom{n-j-\ell-1}{k-j-1}}=\frac{n-j-\ell}{k(k-j)}\geq 1,
\end{align}
by \eqref{ineq-1} we have
\[
\gamma_\ell(\hf) \leq \sum_{r\leq j\leq k} \frac{|\hb^{(j)}(\overline{U})|}{k^{j-\ell-1}} k^{r-\ell-1}\binom{n-r-\ell}{k-r}\leq r(t-1)(r-1)^{\ell-1} k^{r-\ell-1}\binom{n-r-\ell}{k-r}.
\]
Note also that for $\ell+2\leq r\leq k-1$ and $n\geq ek^2\geq  \frac{(\ell+3)(\ell+2)^{\ell-2}}{(\ell+1)^{\ell-1}}k^2$,
\[
\frac{r(r-1)^{\ell-1} k^{r-\ell-1}\binom{n-r-\ell}{k-r}}{(r+1)r^{\ell-1} k^{r-\ell}\binom{n-r-\ell-1}{k-r-1}}= \frac{(r-1)^{\ell-1}(n-r-\ell)}{(r+1)r^{\ell-2}k(k-r)}\geq \frac{(\ell+1)^{\ell-1}(n-r-\ell)}{(\ell+3)(\ell+2)^{\ell-2}k(k-r)} \geq 1.
\]
Then
\begin{align}\label{ineq-key1}
\gamma_\ell(\hf) &\leq (\ell+2)(t-1)(\ell+1)^{\ell-1}k\binom{n-2\ell-2}{k-\ell-2}\leq (\ell+2)(\ell+1)^{\ell}k\binom{n-2\ell-2}{k-\ell-2}.
\end{align}
By $n\geq \frac{(\ell+2)(\ell+1)^\ell -(\ell+1)\ell^{\ell}}{m_{\ell}(\ell+1)}k^2$, we conclude that
\[
\gamma_\ell(\hf)\leq m_{\ell}(\ell+1)\binom{n-2\ell-1}{k-\ell-1}+ (\ell+1)\ell^{\ell}k\binom{n-2\ell-2}{k-\ell-2}.
\]

{\bf Case 2.} $r=\ell+1=t$.

By Lemma  \ref{lem-0} (ii) $\hb^{(\ell+1)}(\overline{U})=\gamma_{\ell}(\hb^{(\ell+1)})$. Then
\begin{align*}
\gamma_\ell(\hf) \leq |\hf(\overline{U})|&\leq \gamma_{\ell}(\hb^{(\ell+1)})\binom{n-2\ell-1}{k-\ell-1}+\sum_{\ell+2\leq j\leq k} |\hb^{(j)}(\overline{U})|\binom{n-j-\ell}{k-j}.
\end{align*}
By \eqref{ineq-2-2} and \eqref{ineq-1} we infer
\begin{align}\label{ineq-key2}
\sum_{\ell+2\leq j\leq k} |\hb^{(j)}(\overline{U})|\binom{n-j-\ell}{k-j}&=\sum_{\ell+2\leq j\leq k} \frac{|\hb^{(j)}(\overline{U})|}{k^{j-\ell-1}}k^{j-\ell-1}\binom{n-j-\ell}{k-j} \nonumber\\[5pt]&\leq \sum_{\ell+2\leq j\leq k}  \frac{|\hb^{(j)}(\overline{U})|}{k^{j-\ell-1}}k\binom{n-2\ell-2}{k-\ell-2}\nonumber\\[5pt]
  &\leq (\ell+1)\ell^{\ell}k\binom{n-2\ell-2}{k-\ell-2}.
\end{align}
By \eqref{ineq-2-2} we conclude that
\begin{align*}
\gamma_\ell(\hf) \leq \gamma_{\ell}(\hb^{(\ell+1)})\binom{n-2\ell-1}{k-\ell-1}+ (\ell+1)\ell^{\ell}k\binom{n-2\ell-2}{k-\ell-2}.
\end{align*}
\end{proof}

\section{Double diversity}

The following intersecting family $\hht_0\subset \binom{[6]}{3}$ was defined in \cite{Furedi}:
\[
\hht_0=\{(1,2,3), (1,2,4),(3,4,5),(3,4,6),(1,5,6),(2,5,6),(1,3,5),(2,4,5),(1,4,6),(2,3,6)\}.
\]
One has $|\hht_0|=10$, $|\hht_0(x)|=5$ for all $1\leq x\leq 6$ and $|\hht_0(x,y)|=2$ for all $1\leq x<y\leq 6$. Note that these imply $|\hht_0(\overline{x},\overline{y})|=|\hht_0|-|\hht_0(x)|-|\hht_0(y)|+|\hht_0(x,y)|=2$.

Noting $|\hht_0|=\frac{1}{2}\binom{6}{3}$ it is easy to see that $\hht_0$ is isomorphic to its complement $\binom{[6]}{3}\setminus \hht_0$, which also equals $\{[6]\setminus T\colon T\in \hht_0\}$. Hence $\hht_0$ is saturated. From $\gamma_2(\hht_0)=2$ it follows also that each 4-sets $V\subset [6]$ contains at least one edge from both $\hht_0$ and its complement.

\begin{example}
Define
\[
\hf_{\hht_0} =\left\{F\in \binom{[n]}{k}\colon F \supset T\mbox{ for some } T\in \hht_0\right\}.
\]
Using $|\hht_0(\overline{P})|=2$ for all $P\in \binom{[6]}{2}$ we infer
\[
\gamma_2(\hf_{\hht_0}) = 2\binom{n-5}{k-3}-\binom{n-6}{k-4}.
\]
\end{example}

The following lemma plays a central role in the proof of Theorem \ref{main-1}. It shows in a stronger way that $m_2(3)=2$.

\begin{lem}\label{lem-1}
Let $\emptyset \neq \hht\subset \binom{[n]}{3}$ be an intersecting family. Then either $\hht$ is isomorphic to one of $\hl$ and $\hht_0$ or there exists $S\subset T\in \hht$ with $|S|=2$ and $|\hht(\overline{S})|\leq 1$.
\end{lem}

\begin{proof}
Fix $R\subset \binom{[n]}{2}$ such that $|\hht(R)|$ is maximized. We distinguish three cases.

{\bf Case 1. } $|\hht(R)|\geq 3$.

Choose $x_1,x_2,x_3$ so that $R\cup \{x_i\}\in \hht$, $1\leq i\leq 3$.  Since $\hht$ is intersecting $\{x_1,x_2,x_3\}\subset T$ for all $T\in \hht(\overline{R})$. Hence $|\hht(\overline{R})|\leq 1$.

{\bf Case 2. } $|\hht(R)|= 2$.

For definiteness let $R=\{r_1,r_2\}$ and suppose that $P=\{p_1,p_2\}$ where $R\cup \{p_i\}\in \hht$, $i=1,2$. By the intersection property  every $T\in \hht(\overline{R})$ satisfies $P\subset T$. If $|\hht(\overline{R})|\leq 1$ then we are done. Thus we assume $Q=\{q_1,q_2\}$ and $P\cup\{q_i\}\in \hht$ for $i=1,2$. In the same way we may assume that $\hht(\overline{P})=\{Q\cup \{s_1\},Q\cup \{s_2\}\}$ for appropriate $s_1,s_2$. Using $(R\cup \{p_i\})\cap (Q\cup \{s_j\})\neq \emptyset$, $1\leq i,j\leq 2$ we infer $\{s_1,s_2\}=R$. That is, we found six of the edges of $\hht_0$, namely $R\cup \{q_i\}$, $Q\cup \{p_i\}$, $P\cup \{r_i\}$, $i=1,2$.

Let $\hf_0$ be the family of these six edges. Choose a pair $S\in \binom{P\cup Q\cup R}{2}$. Note that $|\hf_0(S)|=|\hf_0(\overline{S})|$ and this common value is 2 for $S=P, Q, R$ and it is 1 for the remaining $15-3=12$ pairs.

Assuming $\gamma_2(\hht)\geq 2$, $|\hht\setminus \hf_0|\geq \frac{12}{3}=4$ follows. The intersecting property implies $|T\cap P|=|T\cap Q|=|T\cap R|=1$ for all $T\in \hht \setminus \hf_0$.

Also, if $|T\cap T'|=2$ for $T,T'\in \hht\setminus \hf_0$ then $|\hht(T\cap T')|=2+1=3$ follows, bringing us back to Case 1. Assume by symmetry that $\{r_1,q_1,p_1\}, \{r_1,q_2,p_2\}\in \hht\setminus\hf_0$. Then the only possibility for the remaining edges is $\{r_2,q_1,p_2\},\{r_2,q_2,p_1\}$ whence $\hht$ is isomorphic to $\hht_0$.

{\bf Case 3.} $|T\cap T'|=1$ for all distinct pairs $T,T'\in \hht$.

Set $X=\mathop{\cup}\limits_{T\in \hht} T$. We claim that $|\hht(x)|=3$ for all $x\in X$. Indeed $|\hht(x)|\geq 4$ and the intersecting property would force $x\in T$ for all $T\in \hht$ whence $\hht(\overline{S})=\emptyset$ for all $2$-sets $S$ containing $x$.

If $|\hht(x)|\leq 2$ and $x\in T\in \hht$ then $|\hht(\overline{T\setminus \{x\}})\leq 1$. Thus we may assume that $|\hht(x)|=3$ for all $x\in X$. By symmetry assume that $(1,2,3),(1,4,5),(1,6,7)\in \hht$. Again by symmetry the remaining two edges containing 2 are WLOG $(2,4,6)$ and $(2,5,7)$. To maintain the property $|T\cap T'|=1$ the only possibility for the remaining two edges containing 3 is $(3,4,7)$, $(3,5,6)$. Thus $\hht$ is isomorphic to $\hl$.
\end{proof}

\begin{remark}\label{remark-1}
Recall that a hypergraph $\hh$ is called {\it 3-chromatic} if every set $X$ satisfying $X\cap H\neq \emptyset$ for all $H\in \hh$ contains at least one edge of $\hh$. As we showed above, $\hht_0$ is 3-chromatic. The same is true for $\hl$.
\end{remark}

Our next result shows that Lemma \ref{lem-1} implies $m_1(3)=5$.

\begin{cor}
Let $\hht\subset \binom{[n]}{3}$ be an intersecting family with $\tau(\hht)=3$. Then
\[
\gamma_1(\hht)\leq 5
\]
with equality holding if and only if $\hht=\hht_0$ up to isomorphism.
\end{cor}

\begin{proof}
It is easy to check that $\gamma_1(\hht_0)=5$ and $\gamma_1(\hl)=4$. Thus by Lemma \ref{lem-1} we may assume that $\gamma_2(\hht)\leq 1$. Fix $x,y$ with $|\hht(\overline{x},\overline{y})|\leq 1$ and assume indirectly $|\hht(\overline{v})|\geq 5$ for all $v\in [n]$. Using $|\hht(\overline{x})|=|\hht(\overline{x},y)|+|\hht(\overline{x},\overline{y})|$ we infer $|\hht(\overline{x},y)|\geq 4$ and similarly $|\hht(x,\overline{y})|\geq 4$. Now $\hht(x,\overline{y})$, $\hht(\overline{x},y)$ are cross-intersecting 2-graphs. Hence for an appropriate $z$, $z$ is a common vertex to all edges in $\hht(\overline{x},y)\cup \hht(x,\overline{y})$. In particular, we found (at least) four edges $\{x,z,u_i\}\in \hht$, $1\leq i\leq 4$. The intersecting property implies $\hht(\overline{x},\overline{z})=\emptyset$, i.e., $\tau(\hht)\leq 2$, a contradiction.
\end{proof}

\begin{lem}\label{lem-30}
  Suppose that $\hf\subset \binom{[n]}{k}$ is a saturated intersecting family, $\hb=\hb(\hf)$, $t=\tau(\hf)= 4$. If $\tau(\hb^{(4)})\geq 3$ then for any  $V=\{v_1,v_2\}\subset B_1\in \hb^{(4)}$,
\begin{align}\label{ineq-31}
\sum_{4\leq j \leq k} \frac{|\hb^{(j)}(\overline{V})|}{k^{j-3}}\leq 24.
\end{align}
\end{lem}

\begin{proof}
We prove \eqref{ineq-31} by refining the branching process in Lemma \ref{lem-0}. Assume that  $B_1=\{v_1,v_2,w_{1},w_{2}\}$.  By $\tau(\hb^{(4)})\geq 3$, there exists $B_2\in \hb^{(4)}$ such that $B_2\cap \{w_{1},w_{2}\}=\emptyset$. Since $B_2\cap B_1\neq \emptyset$, we infer $B_2\cap V\neq \emptyset$.

At the first stage, define  sequences $(w_{1}), (w_{2})$ and assign the weight $\frac{1}{2}$ to each of them. At the second stage, we replace each $(w_{i})$, $i=1,2$ by $|B_2\setminus V|$ $2$-sequences of the form $(w_{i},y)$ with $y\in B_1\setminus V$ and weight $\frac{1}{2|B_2\setminus V|}$. At the third stage, since $\tau(\hb^{(4)})\geq 3$, for each sequence $S=(x_1,x_2)$ we may choose $B\in \hb^{(4)}$ such that $\widehat{S}\cap B= \emptyset$. Then we replace $S=(x_1,x_2)$ by $|B\setminus V|$ $3$-sequences of the form $(x_1,x_2,y)$ with $y\in B\setminus V$ and weight $\frac{w(S)}{|B\setminus V|}$. In the subsequent stages, we follow the same procedure as in the proof of Lemma \ref{lem-0}.

By Claim \ref{claim-0}, we see that $|\hb^{(j)}(\overline{V})|\leq |\hs^{(j)}|$ for $j\geq 4$.  Let $S=(x_1,\ldots,x_j)\in \hs^{(j)}$ and let $S_i=(x_1,\ldots,x_i)$ for $i=1,\ldots,j$. Note that $w(S_1)=\frac{1}{2}$, $w(S_2)\geq \frac{1}{6}$ and $w(S_3)\geq \frac{1}{24}$. For $i\geq 4$ assume that $B_i$ is the selected set when replacing $S_{i-1}$ in the  branching process.
Note that $|B_i\setminus V|\leq k$ for $4\leq i\leq j$. Thus $w(S)\geq \frac{1}{24k^{j-3}}$ and \eqref{ineq-31} follows.
\end{proof}

\begin{proof}[Proof of Theorem \ref{main-1}]
Assume that $\hf\subset \binom{[n]}{k}$ is a saturated intersecting family with $\tau(\hf)\geq 3$. Let $\hb=\hb(\hf)$, $t=\tau(\hf)$ and let $r$ be the smallest integer such that $\tau(\hb^{(\leq r)})\geq 3$. Clearly $r\geq t\geq 3$.

Let $U\subset \binom{[n]}{2}$ be defined in  Lemma  \ref{lem-0} with $\ell=2$.  We distinguish two cases.

{\bf Case 1. } $t= 3$.

 If $r\geq 4$ then by applying \eqref{ineq-key1} with $\ell=2$, $t=3$ and  $n\geq 13k^2$,
\[
\gamma_2(\hf) \leq  |\hf(\overline{U})|\leq 2\times 3\times 4k\binom{n-6}{k-4}<  2\binom{n-5}{k-3}-\binom{n-7}{k-5}.
\]

For $r=3$, by Lemma \ref{lem-1} we have $\gamma_2(\hb^{(3)})\leq 2$. If $\gamma_2(\hb^{(3)})=2$, then by Lemma \ref{lem-1}, $\hb^{(3)}=\hl$ or $\hb^{(3)} =\hht_0$ up to isomorphism. By saturatedness $\hf=\hf_{\hl}$ or $\hf=\hf_{\hht_0}$ (cf. Remark \ref{remark-1}). In the first case we have $\gamma_{2}(\hf) = 2\binom{n-5}{k-3}-\binom{n-7}{k-5}$. In the latter case $\gamma_{2}(\hf) = 2\binom{n-5}{k-3}-\binom{n-6}{k-4}<2\binom{n-5}{k-3}-\binom{n-7}{k-5}$.

If $\gamma_2(\hb^{(3)})= 1$, then by Lemma  \ref{lem-0} (ii), $\hb^{(3)}(\overline{U})=1$ and
\begin{align*}
\gamma_2(\hf) \leq |\hf(\overline{U})|&\leq \binom{n-5}{k-3}+\sum_{4\leq j\leq k} |\hb^{(j)}(\overline{U})|\binom{n-j-2}{k-j}.
\end{align*}
Applying \eqref{ineq-key2} with $\ell=2$ and using $n\geq 13k^2$, we conclude that
\begin{align*}
\gamma_2(\hf)\leq  \binom{n-5}{k-3}+12 k\binom{n-6}{k-4}< 2\binom{n-5}{k-3}-\binom{n-7}{k-5}.
\end{align*}

{\bf Case 2. } $t\geq 4$.

If $r\geq 5$, then  by applying \eqref{ineq-key1} with $\ell=2$ and  $n\geq 13k^2$,
\[
\gamma_2(\hf) \leq  |\hf(\overline{U})|\leq 5\times 4^2k^2\binom{n-7}{k-5}<  2\binom{n-5}{k-3}-\binom{n-7}{k-5}.
\]
If $r=t=4$, then let $V$ be defined in  Lemma  \ref{lem-30}. Note that
\[
\gamma_2(\hf) \leq |\hf(\overline{V})|\leq \sum_{4\leq j \leq k} |\hb^{(j)}(\overline{V})|\binom{n-2-j}{k-j}= \sum_{4\leq j \leq k} \frac{|\hb^{(j)}(\overline{V})|}{k^{j-3}}k^{j-3}\binom{n-2-j}{k-j}.
\]
Since $n\geq 13k^2$ and $j\geq4$ imply $k^{j-3}\binom{n-2-j}{k-j}\leq k\binom{n-6}{k-4}$, by \eqref{ineq-31} we infer
\[
\gamma_2(\hf) \leq k\binom{n-6}{k-4}\sum_{4\leq j \leq k} \frac{|\hb^{(j)}(\overline{V})|}{k^{j-3}} \leq 24 k\binom{n-6}{k-4}< 2\binom{n-5}{k-3}-\binom{n-7}{k-5}.
\]
\end{proof}

\section{Results for triple-diversity}

In this section we prove that $m_3(4)=3$. Since $m_1(2)=1$ and $m_2(3)=2$, one might think that $m_{\ell}(\ell+1)=\ell+1$ in general. However, as we will see in Section 5 $m_{\ell}(\ell+1)>\ell$ for $\ell\geq 4$.

Unlike Lemma \ref{lem-1} we could not determine the cases of equality for $m_3(4)=3$. Let us show that there are at least four non-isomorphic possibilities.

\begin{example}[Finite projective plane of order 3]
Let $\mathds{Z}_{13}=\{0,1,2,\ldots,12\}$ be the  additive group of integers modulo 13. Define the intersecting family
\[
\hl_3=\left\{\{i,i+1,i+3,i+9\}\colon i\in \mathds{Z}_{13} \right\}.
\]
For $n\geq 13$ and $k\geq 4$, let
\[
\hf_{\hl_3} =\left\{F\in \binom{[n]}{k}\colon F \supset L \mbox{ for some } L\in \hl_3\right\}.
\]
Note that $|\hl_3(\overline{T})|\geq 3$ for all $T\in \binom{[13]}{2}$, $|\hl_3(\overline{\{0,1,3\}})|= 3$ and $|L\cap L'|=1$ for all distinct $L,L'\in \hl_3$. It is not hard to check that
\[
\gamma_3(\hf_{\hl_3}) = 3\binom{n-7}{k-4}-3\binom{n-10}{k-7}+\binom{n-13}{k-10}.
\]
\end{example}

\begin{example}
Let $A,B,C,D$ and $E$ be pairwise disjoint sets, $|A|=|B|=|C|=|D|=2$, $|E|=3$. Let
\[
\hh=\left\{X\cup Y\colon X\in \{A,B,C,D\}, Y\in \binom{E}{2}\right\}.
\]
Set $A=\{a_0,a_1\}$, $B=\{b_0,b_1\}$, $C=\{c_0,c_1\}$, $D=\{d_0,d_1\}$. Define
\[
\hg=\left\{\{a_i,b_j,c_k,d_\ell\}\colon i+j+k \text{ is odd }\right\},
\]
clearly $|\hg|=8$, $G\cap G'\cap (A\cup B\cup C)\neq \emptyset$ for all $G,G'\in \hg$. (In particular, $\hg$ is intersecting). As $\hh$ and $\hg$ are cross-intersecting, $\hh\cup \hg$ is intersecting, $|\hh\cup \hg|=20$ and $\gamma_3(\hh\cup \hg)=3$.
\end{example}

\begin{example}
Let $U_1,U_2,U_3$ be pairwise disjoint 3-sets. Define the 4-graph $\hh$ as the collection of all 4-sets that are of the form $P_i\cup P_j$,  where $P_i\in \binom{U_i}{2}$, $P_j\in \binom{U_j}{2}$ and $1\leq i<j\leq 3$. Then $\hh$ is 4-uniform, intersecting with covering number 4 and  $|\hh|=27$. It is easy to check that $\gamma_3(\hh)=3$.
\end{example}

\begin{example}
Let $X=\{x_0,x_1,x_2\}$, $Y=\{y_0,y_1,y_2,y_3,y_4\}$ and $Z=\{z_0,z_1\}$ be pairwise disjoint sets.
Let $\ha=\binom{X}{2}$, $\hb=\{(y_i,y_{i+1})\colon 0\leq i\leq 4\}$, $\hc=\{\{y_i,y_{i+1},y_{i+3}\}\colon 0\leq i\leq 4\}$ with subscripts modulo 5. Define
\[
\hh=\left\{A\cup B\colon A\in \ha,\ B\in \hb\right\} \cup \left\{A\cup Z\colon A\in \ha\right\}\cup \left\{C\cup \{z\}\colon C\in \hc,\ z\in Z\right\}.
\]
Then $|\hh|=3\times 5+3+5\times 2=28$. It is a little tedious but not hard to check that $\gamma_3(\hh)=3$.
\end{example}

The main result of this section is establishing $m_3(4)=3$.

\begin{prop}\label{prop-2-2}
Let $\hf\subset \binom{[n]}{4}$ be an intersecting family. Then $\gamma_3(\hf)\leq 3$.  
\end{prop}

As we will see in Section 5, $m_{\ell}(\ell+1)=\ell$ is no longer true for $\ell\geq 4$. Therefore we cannot expect a short slick proof. As a matter of fact we need to distinguish several cases and go through a tedious process.

\begin{fact}\label{fact-1}
Suppose that $\hf\subset \binom{[n]}{\ell}$ is an intersecting family and $x\in \cup \hf$. Then
\begin{align}
|\hf(x)| \geq \gamma_{\ell-1}(\hf) +1
\end{align}
with equality only if $\hf(x)$ consists of $\gamma_{\ell-1}(\hf)+1$ pairwise disjoint $(\ell-1)$-sets.
\end{fact}

\begin{proof}
Fix $F\in \hf$ with $x\in \hf$. Then
\[
 \gamma_{\ell-1}(\hf)\leq |\hf(\overline{F\setminus \{x\}})|\leq |\hf(x)\setminus \{F\setminus \{x\}\}|=|F(x)|-1.
\]
Suppose that there exist $F,F'\in \hf$ with $x\in F\cap F'$ and $|F\cap F'|\geq 2$. This implies $|\hf(\overline{F\setminus \{x\}})|\leq |\hf(x)\setminus\{F\setminus \{x\}, F'\setminus\{x\}\}|$, proving $|\hf(x)|\geq \gamma_{\ell-1}(\hf)+2$.
\end{proof}

\begin{lem}\label{lem-2-0}
Suppose that $\hf\subset \binom{[n]}{4}$ satisfies $\gamma_3(\hf)\geq 3$ and $|F\cap F'|=1$ for all distinct $F,F'\in \hf$. Then $\hf$ is isomorphic to $\hl_3$.
\end{lem}

\begin{proof}
In view of Fact \ref{fact-1}, $|\hf(x)|\geq 4$ for all $x\in \hf$. Let $F\in \hf$. Then
\[
\left|\cup \hf\right|=1+\sum_{x\in F}(|\hf(x)|-1)\geq 1+4\times 3 =13.
\]
 Using the identity
\[
4|\hf|=\sum_{x\in \cup \hf} |\hf(x)| \geq 4\times 13,
\]
$|\hf|\geq 13=4^2-4+1$. By Deza's Theorem \cite{deza}, $\hf$ is isomorphic to $\hl_3$.
\end{proof}

The next lemma shows $\gamma_3(\hf)\leq 3$ for another important special case.

\begin{lem}\label{lem-2-1}
Let $\hf\subset \binom{[n]}{4}$ be an intersecting family with $\Delta_3(\hf)=2$. Then $\gamma_3(\hf)\leq 3$.
\end{lem}
\begin{proof}
Assume that $\tau(\hf)= 4$. Assume indirectly that $|\hf(\overline{T})|\geq 4$ for all $T\in \binom{[n]}{3}$.

\begin{claim}\label{claim-2}
For every $P\in \binom{[n]}{2}$, $\hf(P)$ is a 2-graph with maximum degree at most 2.
\end{claim}
\begin{proof}
Suppose that there is a vertex $u$ having degree 3 in $\hf(P)$. Then $|\hf(P\cup \{u\})|\geq 3$, contradicting $\Delta_3(\hf)=2$.
\end{proof}

By $\Delta_3(\hf)=2$, there exists
$T\in \binom{[n]}{3}$ with $T\cup \{x\}$, $T\cup\{y\}\in \hf$. Then $|\hf(\overline{T})|\geq 4$ implies that $|\hf(\overline{T},\{x,y\})|\geq 4$. By Claim \ref{claim-2}, $\hf(\overline{T},\{x,y\})$ has matching number at least 2. We claim that  $\nu(\hf(\overline{T},\{x,y\}))\leq 3$. Indeed, if there are pairwise disjoint 2-sets $E_1,E_2,E_3,E_4$ such that $E_i\cup\{x,y\}\in \hf(\overline{T})$, then for any $F\in \hf(\overline{x},\overline{y})$, $F\cap E_i\neq \emptyset$, $i=1,2,3,4$. It follows that $F\cap (T\cup \{x\})= \emptyset$, contradicting the fact that $\hf$ is intersecting. Thus $\nu(\hf(\overline{T},\{x,y\}))\leq 3$.

Let us  distinguish two cases.

\vspace{5pt}
{\bf Case 1.} $\nu(\hf(\overline{T},\{x,y\}))= 3$.

Let $E_1=(a_0,a_1),E_2=(b_0,b_1),E_3=(c_0,c_1)\in \hf(\overline{T},\{x,y\})$. Then  for any $F\in \hf(\overline{x},\overline{y})$, $F\cap E_i\neq \emptyset$, $i=1,2,3$ and $F\cap T\neq \emptyset$.  By $|\hf(\overline{T},\{x,y\})|\geq 4$ there exists $E_4\in \hf(\overline{T},\{x,y\})$. By symmetry assume that $E_4=(a_0,d)$ for some appropriate $d$. If $d\notin E_1\cup E_2\cup E_3$, then  for every $F\in \hf(\overline{\{x,y,a_0\}})$, $\{a_1,d\}\subset F$ and $F\cap E_i\neq \emptyset$, $i=2,3$. But then $F\cap (T\cup\{x\})= \emptyset$, contradicting the fact that $\hf$ is intersecting. Thus by symmetry we may assume that $E_4=(a_0,b_1)$.

Let $\hg=\hf(\overline{x},\overline{y})$. Then $|\hf(\overline{\{x,y,a_0\}})|\geq 4$ implies $|\hg(a_1,b_1)|\geq 4$ and $|\hf(\overline{\{x,y,b_1\}})|\geq 4$ implies $|\hg(a_0,b_0)|\geq 4$. Moreover, the intersection property implies  $\hg(a_0,b_0),\hg(a_1,b_1)\subset \{c_0,c_1\}\times T$.  Since $\hg(a_0,b_0),\hg(a_1,b_1)$ are  cross-intersecting 2-graphs and $|\hg(a_0,b_0)|\geq 4$, $|\hg(a_1,b_1)|\geq 4$, both $\hg(a_0,b_0)$ and $\hg(a_1,b_1)$ have to be stars, contradicting Claim \ref{claim-2}.

{\bf Case 2.} $\nu(\hf(\overline{T},\{x,y\}))= 2$.

By Claim \ref{claim-2}, each component of $\hf(\overline{T},\{x,y\})$ is a cycle or a path. As $\nu(\hf(\overline{T},\{x,y\}))= 2$ and $|\hf(\overline{T},\{x,y\})|\geq 4$, $\hf(\overline{T},\{x,y\})$ contains one of the following four subgraphs: a triangle plus an edge, two disjoint paths of length 2, a path of length 4, a cycle of length 4. Let $T=\{z_1,z_2,z_3\}$. We deal with the 4 cases separately.

{\bf Case 2.1.} $\hf(\overline{T},\{x,y\})$ contains a triangle plus an edge.

 Assume that $\{(a,b),(a,c),(b,c), (d,e)\}\subset \hf(\overline{T},\{x,y\})$. For every $F\in \hf(\overline{\{x,y,c\}})$, by the intersection property $\{a,b\}\subset F$, $F\cap \{d,e\}\neq \emptyset$ and $F\cap T\neq \emptyset$ . By $|\hf(\overline{\{x,y,c\}})|\geq 4$ and Claim \ref{claim-2},   $|\hf(\{a,b,d\})|= 2= |\hf(\{a,b,e\})|$. It follows that
 there exists $z\in T$ such that $\{a,b,d,z\}$, $\{a,b,e,z\}\in \hf$.

 Now consider $\hf(\overline{\{a,b,z\}})$. As $\{a,b,d,z\},\{a,b,e,z\}\in \hf$, we infer  $\{d,e\}\subset F$ for all $F\in \hf(\overline{\{a,b,z\}})$. Moreover, $F\cap\{a,b,x,y\}\neq \emptyset$ implies that $F\cap\{x,y\}\neq \emptyset$. Since $\{x,y,d,e\}\in \hf$ and $\Delta_3(\hf)=2$, we conclude that
 $|\hf(\overline{\{a,b,z\}})|=|\hf(\{d,e,x\})|+|\hf(\{d,e,y\})|-1\leq 3$, contradicting our assumption.

{\bf Case 2.2.} $\hf(\overline{T},\{x,y\})$ contains two disjoint paths of length 2.

Assume that $\{(a,b),(a,c),(d,e), (d,f)\}\subset \hf(\overline{T},\{x,y\})$. Then  for every  $F\in \hf(\overline{\{x,y,a\}})$, $\{b,c\}\subset F$, $F\cap\{d,e\}\neq \emptyset$, $F\cap\{d,f\}\neq \emptyset$ and $F\cap T\neq \emptyset$. It follows that $\{b,c,d\}\subset F$. By $\Delta_3(\hf)= 2$, we know that $|\hf(\overline{\{x,y,a\}})|\leq |\hf(\{b,c,d\})|\leq 2$, contradicting $\gamma_3(\hf)\geq 4$.

{\bf Case 2.3.} $\hf(\overline{T},\{x,y\})$ contains a path of length 4.

Assume that $\{(a,b),(b,c),(c,d), (d,e)\}\subset \hf(\overline{T},\{x,y\})$. Then for every  $F\in \hf(\overline{\{x,y,b\}})$, $\{a,c\}\subset F$, $F\cap \{d,e\}\neq\emptyset$ and $F\cap T\neq \emptyset$.  Thus, $\hf(\overline{\{x,y,b\}},\{a,c\})\subset\{d,e\}\times T$. Similarly, we have $\hf(\overline{\{x,y,d\}},\{c,e\})\subset\{a,b\}\times T$.

By symmetry, assume that $\{a,c,d,z_1\},\{a,c,d,z_2\}\in\hf$. Then for any $F\in \hf(\overline{\{a,c,d\}})$ we infer $\{z_1,z_2\}\subset F$ and $F\cap \{x,y\}\neq \emptyset$. We claim that $\{z_1,z_2,x,y\}\notin \hf$. Otherwise, by $\Delta_3(\hf)=2$  we have $|\hf(\overline{\{a,c,d\}})|\leq |\hf(\{z_1,z_2,x\})|+ |\hf(\{z_1,z_2,y\})|-1\leq 3$, a contradiction. Thus $\{z_1,z_2,x,y\}\notin \hf$. By symmetry assume that for appropriate $u,v$,
\begin{align}\label{ineq-3-1}
\{z_1,z_2,x,z_3\},\{z_1,z_2,x,u\},\{z_1,z_2,y,z_3\},\{z_1,z_2,y,v\}\in \hf(\overline{\{a,c,d\}}).
\end{align}
If $u\notin \{a,b,c,d,e\}$, then for $F\in\hf(\overline{\{z_1,z_2,x\}})$, $\{z_3,u\}\subset F$ and $F\cap (\{x,y\}\cup P)\neq \emptyset$ for all $P\in \{(a,b),(b,c),(c,d), (d,e)\}$.  It follows that either $y\in F$ or $\{b,d\}\subset F$. Consequently, $|\hf(\overline{\{z_1,z_2,x\}})|\leq |\hf(\{z_3,u,y\})|+1\leq 3$, a contradiction. Thus $u\in \{b,e\}$. By the same argument we have $v\in \{b,e\}$.

If $\{b,c,e,z_1\},\{b,c,e,z_2\}\in \hf$, then for every $F\in \hf(\overline{\{b,c,e\}})$, $\{z_1,z_2\}\subset F$ and $F\cap \{x,y,b,c\}\neq \emptyset$. It follows that
\[
|\hf(\overline{\{b,c,e\}})|\leq |\hf(\{z_1,z_2,x\})|+|\hf(\{z_1,z_2,y\})|.
\]
Note that $u,v\in \{b,e\}$ implies that $\{z_1,z_2,x,u\},\{z_1,z_2,y,v\}\notin \hf(\overline{\{b,c,e\}})$. By \eqref{ineq-3-1}, we infer $|\hf(\overline{\{b,c,e\}})|\leq 2$, a contradiction. Thus we cannot have  both $\{b,c,e,z_1\}\in \hf$ and $\{b,c,e,z_2\}\in \hf$. By symmetry, assume that $\{b,c,e,z_2\},\{b,c,e,z_3\}\in \hf$. Then by considering $\hf(\overline{\{b,c,e\}})$ we infer that for some $u',v'\in \{a,d\}$,
\begin{align}\label{ineq-3-2}
\{z_2,z_3,x,z_1\},\{z_2,z_3,x,u'\},\{z_2,z_3,y,z_1\},\{z_2,z_3,y,v'\}\in \hf(\overline{\{b,c,e\}}).
\end{align}

Recall that  $\hf(\overline{\{x,y,b\}},\{a,c\})\subset\{d,e\}\times T$ and $|\hf(\overline{\{x,y,b\}},\{a,c\})|\geq 4$. It follows that $|\hf(\overline{\{x,y,b\}},\{a,c,e\})|\geq 2$. Then either  $\{a,c,e,z_1\}\in \hf$ or  $\{a,c,e,z_3\}\in \hf$.  If $\{a,c,e,z_1\}\in \hf$, then by \eqref{ineq-3-2} $u'=v'=a$. Then for every $F\in \hf(\overline{\{z_2,z_3,a\}})$ we infer $\{x,y\}\subset F$. Now
\[
(b,c),(c,d), (d,e)\in \hf(\overline{\{z_2,z_3,a\}},\{x,y\}).
\]
If $\nu(\hf(\overline{\{z_2,z_3,a\}},\{x,y\}))\geq 3$, then by the same argument as in Case 1 we are done. Thus we may assume that  $\nu(\hf(\overline{\{z_2,z_3,a\}},\{x,y\}))= 2$. Note that $b,c,d$ have already degree 2 in $\hf(x,y)$. By Claim \ref{claim-2}, there is an extra edge $(e,z)\in \hf(\overline{\{z_2,z_3,a\}},\{x,y\})$. Now we claim that $z=z_1$. Otherwise, $\{(a,b),(b,c),(c,d), (d,e),(e,z)\}\subset \hf(\overline{T},\{x,y\})$, contradicting $\nu(\hf(\overline{T},\{x,y\}))=2$. Thus $(x,y,e,z_1)\in \hf$.  But then $\{a,c,d,z_2\}\cap \{x,y,e,z_1\}=\emptyset$, a contradiction.

Similarly, if $\{a,c,e,z_3\}\in \hf$ then by \eqref{ineq-3-1} $u=v=e$. Then for every $F\in \hf(\overline{\{z_1,z_2,e\}})$ we infer $\{x,y\}\subset F$.  Note that   $|\hf(\overline{\{z_1,z_2,e\}})|\geq 4$, $\nu(\hf(\overline{\{z_1,z_2,z_3\}},\{x,y\}))= 2$ and $\nu(\hf(\overline{\{z_1,z_2,e\}},\{x,y\}))=2$. By Claim \ref{claim-2} we infer $\{x,y,a,z_3\}\in \hf$. But then $\{b,c,e,z_2\}\cap \{x,y,a,z_3\}=\emptyset$, contradiction again.

{\bf Case 2.4.} $\hf(\overline{T},\{x,y\})$ contains a cycle of length 4.

\begin{claim}\label{claim-3}
For any $R\in \binom{[n]}{3}$ with $R\cup \{u\}$, $R\cup \{v\}\in \hf$, $\hf(\overline{R},\{u,v\})$ is a $C_4$.
\end{claim}

\begin{proof}
Indeed, otherwise we are done by one of the previous cases.
\end{proof}

Assume that $\{(a_1,b_1),(a_1,b_2),(a_2,b_1), (a_2,b_2)\}\subset \hf(\overline{T},\{x,y\})$.  Then $|\hf(\overline{\{x,y,a_i\}})|\geq 4$ and $|\hf(\overline{\{x,y,b_j\}})|\geq 4$, $i, j=1,2$. By Claim \ref{claim-3} we infer that  $\hf(\overline{\{x,y,a_i\}},\{b_1,b_2\})$, $\hf(\overline{\{x,y,b_j\}},\{a_1,a_2\})$ are all cycles of length 4. 

\begin{claim}\label{claim-4}
$\hf(\overline{\{x,y,a_2\}},\{b_1,b_2,a_1\})=\emptyset=\hf(\overline{\{x,y,b_2\}},\{a_1,a_2,b_1\})$.
\end{claim}

\begin{proof}
Suppose that $\hf(\overline{\{x,y,a_2\}},\{b_1,b_2,a_1\})$ is non-empty. Since $\hf(\overline{\{x,y,a_2\}},\{b_1,b_2\})$ is a $C_4$, $a_1$ has degree 2 in $\hf(\overline{\{x,y,a_2\}},\{b_1,b_2\})$.  By symmetry we may assume that $\{b_1,b_2,a_1,z_1\}$, $\{b_1,b_2,a_1,z_2\}\in \hf$. Then consider $F\in \hf(\overline{\{b_1,b_2,a_1\}})$. We infer $\{z_1,z_2\}\subset F$ and $F\cap \{x,y\}\neq \emptyset$. If $\{z_1,z_2,x,y\}\in \hf$ then by $\Delta_3(\hf)=2$ we have $|\hf(\overline{\{b_1,b_2,a_1\}})|=|\hf(\{z_1,z_2,x\})|+|\hf(\{z_1,z_2,y\})|-1\leq 3$, contradicting our assumption. Thus $\{z_1,z_2,x,y\}\notin \hf$.

Since $\{z_1,z_2,x,z_3\},\{z_1,z_2,y,z_3\}\in\hf$ and by Claim \ref{claim-3} $\hf(\overline{\{b_1,b_2,a_1\}},\{z_1,z_2\})$ is a $C_4$, there exists some $u$ such that $\{z_1,z_2,x,u\},\{z_1,z_2,y,u\}\in\hf$. If $u\neq a_2$, then for each $F\in \hf(\overline{\{x,z_1,z_2\}})$, $\{u,z_3\}\subset F$ and $F\cap\{x,y,a_i,b_j\}\neq \emptyset$, $i,j\in \{1,2\}$. It follows that either $y\in F$ or $F=\{u,z_3,a_1,a_2\}$ or $F=\{u,z_3,b_1,b_2\}$. But then $\hf(\overline{\{x,z_1,z_2\}},\{u,z_3\})$ cannot be a $C_4$, contradicting Claim \ref{claim-3}. Thus we must have $u=a_2$.

Now $\{z_1,z_2,x,a_2\},\{z_1,z_2,y,a_2\}\in\hf$. Then by Claim \ref{claim-3} $\hf(\overline{\{z_1,z_2,a_2\}},\{x,y\})$ is a $C_4$. However,  Claim \ref{claim-2} implies that $\{(a_1,b_1),(a_1,b_2),(a_2,b_1), (a_2,b_2)\}$ is a connected component of $\hf(x,y)$. It follows that $\{(a_1,b_1),(a_1,b_2)\}$ is a connected component of $\hf(\overline{\{z_1,z_2,a_2\}},\{x,y\})$, contradicting Claim \ref{claim-3} again.
Thus, $\hf(\overline{\{x,y,a_2\}},\{b_1,b_2,a_1\})=\emptyset$. By a similar argument, we can show
$\hf(\overline{\{x,y,b_2\}},\{a_1,a_2,b_1\})=\emptyset$ as well.
\end{proof}

By Claim \ref{claim-4}, we infer that $\hf(\overline{\{x,y,a_2\}},\{b_1,b_2\})$, $\hf(\overline{\{x,y,b_2\}},\{a_1,a_2\})$ are cross-intersecting 2-graphs. However, by Claim \ref{claim-3} they are both cycles of length 4, the final contradiction.
\end{proof}

\begin{proof}[Proof of Proposition \ref{prop-2-2}]
Assume that $\tau(\hf)= 4$. Assume indirectly that $|\hf(\overline{T})|\geq 4$ for all $T\in \binom{[n]}{3}$. 
First we claim that $\Delta_3(\hf)= 1$. Indeed, if there exists $T\in \binom{[n]}{3}$ such that $|\hf(T)|\geq 4$ then $|\hf(\overline{T})|\leq 1$, contradicting our assumption.  If $\Delta_3(\hf)= 3$, then there exists $T\in \binom{[n]}{3}$ such that $T\cup \{x_i\}\in \hf$, $i=1,2,3$. It follows that  $|\hf(\overline{T})|\leq |\hf(\{x_1,x_2,x_3\})|\leq \Delta_3(\hf)= 3$, a contradiction. If $\Delta_3(\hf)=2$, then by Lemma \ref{lem-2-1} $\gamma_3(\hf)\leq 3$. Thus $\Delta_3(\hf)=1$.

Note that $\Delta_3(\hf)=1$ implies that $\hf(P)$ consists of pairwise disjoint 2-sets  for every $P\in \binom{[n]}{2}$.  If $|\hf(P)|\geq 5$ for some $P\in \binom{[n]}{2}$ then $\hf(\overline{P})=\emptyset$. If $\Delta_2(\hf)=1$, then  $|F\cap F'|=1$ for all distinct $F,F'\in \hf$ and by Lemma \ref{lem-2-0} $\gamma_3(\hf)\leq 3$. Thus  $2\leq \Delta_2(\hf)\leq 4$. Now we distinguish three cases.

\vspace{5pt}
{\bf Case 1.} $\Delta_2(\hf)= 4$.

Let $P=\{x,y\}$ and $P\cup \{x_{i0},x_{i1}\}\in \hf$, $i=1,2,3,4$. Then $|F\cap \{x_{i0},x_{i1}\}|=1$ for any $F\in \hf(\overline{P})$. Moreover, $|\hf(\overline{P\cup \{x_{i,1-j}\}})|\geq 4$ implies $|\hf(\overline{P}, \{x_{ij}\})|\geq 4$, $i=1,2,3,4$, $j=0,1$. Note that  $\Delta_3(\hf)= 1$ implies that $\hf(\overline{P}, \{x_{10}\})\cap \hf(\overline{P},\{x_{11}\})=\emptyset$. It follows that $|\hf(\overline{P}, \{x_{10}\})\cup \hf(\overline{P},\{x_{11}\})|\geq 8$. That is, $\hf(\overline{P}, \{x_{10}\})\cup \hf(\overline{P},\{x_{11}\})=\{x_{20},x_{21}\}\times \{x_{30},x_{31}\}\times \{x_{40},x_{41}\}$. As $\hf(\overline{P}, \{x_{10}\}), \hf(\overline{P},\{x_{11}\})$ are cross-intersecting, by symmetry we may assume that $\{x_{20},x_{30},x_{40}\}, \{x_{21},x_{31},x_{41}\}\in \hf(\overline{P}, \{x_{10}\})$. Then by $|\hf(\overline{P}, \{x_{10}\})|\geq 4$, there exists another $E\in \hf(\overline{P}, \{x_{10}\})$. Then either $|E\cap \{x_{20},x_{30},x_{40}\}|=2$ or $|E\cap \{x_{21},x_{31},x_{41}\}|=2$. It follows that  $|(E\cup \{x_{10}\})\cap \{x_{10},x_{20},x_{30},x_{40}\}|=3$ or $|(E\cup \{x_{10}\})\cap \{x_{10},x_{21},x_{31},x_{41}\}|=3$, contradicting $\Delta_3(\hf)= 1$ .

\vspace{5pt}
{\bf Case 2.}  $\Delta_2(\hf)= 3$.

Let $P=\{x,y\}$ and $P\cup \{x_{i0},x_{i1}\}\in \hf$, $i=1,2,3$. Then $|F\cap \{x_{i0},x_{i1}\}|\geq 1$ for any $F\in \hf(\overline{P})$. By $|\hf(\overline{\{x,y,x_{11}\}})|\geq 4$ and $\Delta_3(\hf)= 1$, we infer that there exist appropriate $z_0,z_1,z_2,z_3$ such that
\begin{align}\label{ineq-key6}
\{x_{20},x_{30},z_0\}, \{x_{20},x_{31},z_1\},\{x_{21},x_{30},z_2\},\{x_{21},x_{31},z_3\}\in \hf(\overline{\{x,y,x_{11}\}},\{x_{10}\}).
\end{align}
We claim that $z_0\notin \{x_{20},x_{21},x_{30},x_{31}\}$. Indeed, if $z_0=x_{21}$ then by $\Delta_3(\hf)= 1$ we infer $z_2=x_{20}$ and  $\{x_{10},x_{20},x_{30},z_0\}=\{x_{10},x_{21},x_{30},z_2\}$. It implies that $|\hf(\overline{\{x,y,x_{11}\}})|\leq  3$, a contradiction. If $z_0=x_{31}$ then   $\{x_{10},x_{20},x_{30},z_0\}=\{x_{10},x_{20},x_{31},z_1\}$ and $|\hf(\overline{\{x,y,x_{11}\}})|\leq  3$, contradiction again. Thus $z_0\notin \{x_{20},x_{21},x_{30},x_{31}\}$. Similarly $z_1,z_2,z_3\notin \{x_{20},x_{21},x_{30},x_{31}\}$.

Since $\hf(\overline{\{x,y,x_{11}\}}, \{x_{10}\})$ and $\hf(\overline{\{x,y,x_{10}\}}, \{x_{11}\})$ are cross-intersecting,
\begin{align}\label{ineq-key7}
\{x_{21},x_{31},z_0\}, \{x_{21},x_{30},z_1\},\{x_{20},x_{31},z_2\},\{x_{20},x_{30},z_3\}\in \hf(\overline{\{x,y,x_{10}\},\{x_{11}\}}).
\end{align}
We claim that $z_0,z_1,z_2,z_3$ are distinct elements. Indeed, by $\Delta_3(\hf)= 1$ we see that $|T\cap T'|\neq 2$ for all $T,T'\in \hf(u)$ with  $u\in [n]$. Then by \eqref{ineq-key6} it follows that $z_0\neq z_1$, $z_0\neq z_2$, $z_3\neq z_1$ and $z_3\neq z_2$. If $z_0=z_3$ then by \eqref{ineq-key6} and \eqref{ineq-key7} $\{x_{20},x_{30},z_0\}\in \hf(x_{10})\cap \hf(x_{11})$, contradicting $\Delta_3(\hf)= 1$. Hence $z_0\neq z_3$. Similarly $z_1\neq z_3$. Thus $z_0,z_1,z_2,z_3$ are all distinct from each other.

Now consider $F\in \hf(\overline{\{x_{10},x_{11},y\}})$. Clearly $x\in F$ and $F\cap F'\neq \emptyset$ for all $F'\in \hf(\overline{\{x,y\}})$. Then we claim that either $\{x_{20},x_{21}\}\subset F$ or $\{x_{30},x_{31}\}\subset F$. Indeed, otherwise by symmetry assume $x_{20},x_{30}\notin F$. Then $F\cap\{x_{10},x_{20},x_{30},z_0\}\neq \emptyset$, $F\cap\{x_{11},x_{20},x_{30},z_3\}\neq \emptyset$ imply $z_0,z_3\in F$. It follows that at least one of $x_{21}$ and $x_{31}$ is not in $F$. By symmetry assume $x_{21}\notin F$. Then $F\cap \{x_{10},x_{21},x_{30},z_2\}\neq \emptyset$ imply $z_2\in F$. Hence $F=\{x,z_0,z_2,z_3\}$. But then $F\cap \{x_{11},x_{21},x_{30},z_1\}=\emptyset$,  contradiction. Thus, either $\{x_{20},x_{21}\}\subset F$ or $\{x_{30},x_{31}\}\subset F$ for all $F\in \hf(\overline{\{x_{10},x_{11},y\}})$. Recall that $x\in F$. By $\Delta_3(\hf)=1$ we conclude that $|\hf(\overline{\{x_{10},x_{11},y\}})|\leq 2$.

\vspace{5pt}
{\bf Case 3.} $\Delta_2(\hf)= 2$.

Let $P=\{x,y\}$ and $P\cup \{x_{i0},x_{i1}\}\in \hf$, $i=1,2$. Note that $\Delta_2(\hf)= 2$ and $\Delta_3(\hf)= 1$ imply that $\hf(P)$ consists of pairwise disjoint 2-sets and $|\hf(P)|\leq 2$ for all $P\in \binom{[n]}{2}$. Let $T_0=P\cup \{x_{10}\}$ and $T_1=P\cup \{x_{11}\}$. By $|\hf(\overline{T_1})|\geq 4$, we infer $|\hf(\overline{T_1},\{x_{10},x_{20}\})|=|\hf(\overline{T_1},\{x_{10},x_{21}\})|=2$.
 Similarly, $|\hf(\overline{T_0},\{x_{11},x_{20}\})|=|\hf(\overline{T_0},\{x_{11},x_{21}\})|=2$.

\begin{claim}\label{claim-5}
  $\hf(\overline{T_1},\{x_{10},x_{20},x_{21}\})=\emptyset=\hf(\overline{T_0},\{x_{11},x_{20},x_{21}\})$.
\end{claim}
\begin{proof}
If $\hf(\overline{T_1},\{x_{10},x_{20},x_{21}\})\neq \emptyset$, then
\[
\hf(\overline{T_1})=|\hf(\overline{T_1},\{x_{10},x_{20}\})|+|\hf(\overline{T_1},\{x_{10},x_{21}\})|
-|\hf(\overline{T_1},\{x_{10},x_{20},x_{21}\})|\leq 3,
 \]
 contradicting $|\hf(\overline{T_1})|\geq 4$. Thus $\hf(\overline{T_1},\{x_{10},x_{20},x_{21}\})=\emptyset$.
 By the same argument, $\hf(\overline{T_0},\{x_{11},x_{20},x_{21}\})=\emptyset$.
 \end{proof}

 Assume that $(y_0,z_0),(y_1,z_1)\in \hf(\overline{T_1},\{x_{10},x_{20}\})$. By Claim \ref{claim-5}
 $\hf(\overline{T_1},\{x_{10},x_{20}\})$, $\hf(\overline{T_0},\{x_{11},x_{21}\})$ are cross-intersecting. Thus we may assume that  $(y_0,z_1),(y_1,z_0)\in \hf(\overline{T_0},\{x_{11},x_{21}\})$. Similarly, $\hf(\overline{T_1},\{x_{10},x_{21}\})$, $\hf(\overline{T_0},\{x_{11},x_{20}\})$ are cross-intersecting and
  assume that $(y_0',z_0'),(y_1',z_1')\in \hf(\overline{T_1},\{x_{10},x_{21}\})$,   $(y_0',z_1'),(y_1',z_0')\in \hf(\overline{T_0},\{x_{11},x_{20}\})$.

Note that $\Delta_3(\hf)= 1$ implies
\[
\hf(\{x_{1i},x_{20}\})\cap\hf(\{x_{1i},x_{21}\})=\emptyset\mbox{ and } \hf(\{x_{10},x_{2i}\})\cap\hf(\{x_{11},x_{2i}\})=\emptyset,\ i=0,1.
\]
 It follows that
\[
(y_0,z_0),(y_1,z_1),(y_0,z_1),(y_1,z_0),(y_0',z_0'),(y_1',z_1'),(y_0',z_1'),(y_1',z_0')
\]
are all distinct from each other. Therefore $|\{y_0,z_0,y_1,z_1\}\cap \{y_0',z_0',y_1',z_1'\}|\leq 2$.

Now consider $F\in \hf(\overline{\{x_{10},x_{11},y\}})$. Clearly $x\in F$ and $F\cap F'\neq \emptyset$ for all $F'\in \hf(\overline{\{x,y\}})$. By $\Delta_3(\hf)= 1$, there is at most one possibility for $F$ with $\{x_{20},x_{21}\}\subset F$. If $\{y_0,z_0,y_1,z_1\}\cap \{y_0',z_0',y_1',z_1'\}= \emptyset$, then there is no possibility for $F$ with $|F\cap\{x_{20},x_{21}\}|\leq 1$. It follows that $\hf(\overline{\{x_{10},x_{11},y\}})\leq 1$, a contradiction.

If $|\{y_0,z_0,y_1,z_1\}\cap \{y_0',z_0',y_1',z_1'\}|= 1$, by symmetry assume that $y_0=y_0'$, then the only possibility for $F$ with $|F\cap\{x_{20},x_{21}\}|\leq 1$ is $\{x,y_0,y_1,y_1'\}$. It follows that $\hf(\overline{\{x_{10},x_{11},y\}})\leq 2$, a contradiction.

If $|\{y_0,z_0,y_1,z_1\}\cap \{y_0',z_0',y_1',z_1'\}|= 2$, by symmetry there are two possibilities: (i) $y_0=y_0'$ and $y_1=y_1'$; (ii) $y_0=y_0'$ and $y_1=z_0'$. For (i), if $|F\cap\{x_{20},x_{21}\}|\leq 1$ then $\{x,y_0,y_1\}\subset F$. By $\Delta_3(\hf)= 1$, there is at most one possibility for such $F$. Hence  $\hf(\overline{\{x_{10},x_{11},y\}})\leq 2$, a contradiction. For (ii), if $|F\cap\{x_{20},x_{21}\}|\leq 1$ then we also have $\{x,y_0,y_1\}\subset F$  and thereby $\hf(\overline{\{x_{10},x_{11},y\}})\leq 2$, the final contradiction.
\end{proof}

By Theorem \ref{thm-main0} and Proposition \ref{prop-2-2}, Theorem \ref{main-2} follows.

\section{Some bounds on $m_{\ell-1}(\ell)$}

In this section, we establish some upper and lower bounds on $m_{\ell-1}(\ell)$.

Let us prove a simple inequality connecting  $m_0(\ell)$ and $m_{\ell-1}(\ell)$.

\begin{prop}
\begin{align}\label{ineq-key5}
m_{0}(\ell)>\ell m_{\ell-1}(\ell).
\end{align}
\end{prop}

\begin{proof}
Let $\hf$ be an intersecting $\ell$-graph satisfying $\tau(\hf)=\ell$ and $\gamma_{\ell-1}(\hf)= m_{\ell-1}(\ell)$. Choose an arbitrary edge $G\in \hf$ and let $G=\{x_1,x_2,\ldots,x_\ell\}$. Define
\[
\hf_i=\{F\in \hf\colon F\cap G=\{x_i\}\},\ 1\leq i\leq \ell,\ \hf_0=\hf\setminus (\hf_1\cup \ldots\cup \hf_\ell).
\]
Since $\hf$ is intersecting, $\hf(\overline{G\setminus \{x_i\}})=\hf_i$, $1\leq i\leq \ell$. By definition $|\hf_i|\geq \gamma_{\ell-1}(\hf)= m_{\ell-1}(\ell)$. Using
\[
m_0(\ell)\geq |\hf|=|\hf_0|+|\hf_1|+\ldots+|\hf_\ell|,
\]
we infer $m_0(\ell)-|\hf_0|\geq \ell m_{\ell-1}(\ell)$. Since $G\in \hf_0$, \eqref{ineq-key5} follows.
\end{proof}

\begin{prop}
Let $i,j$ be non-negative integers, $1\leq i\leq  j\leq \ell$. Suppose that $\hf\subset \binom{[n]}{\ell}$ is intersecting, $\tau(\hf)\geq j$. Then
\begin{align}\label{ineq-fa}
\Delta_i(\hf) \leq \ell^{j-i}\Delta_j(\hf).
\end{align}
\end{prop}

\begin{proof}
Let us fix $j$ and  prove \eqref{ineq-fa} by induction on $j-i$. For $j=i$ the equality holds trivially.

Given an arbitrary $A\in \binom{[n]}{i}$ with $i<j$. By $\tau(\hf)\geq j>i$, there exists $F_0=\{x_1,\ldots,x_\ell\}\in \hf$ with $F_0\cap A=\emptyset$. For each of the $(i+1)$-sets $A_t=A\cup \{x_t\}$, by the induction hypothesis
 \[
 |\hf(A\cup \{x_t\})|\leq  \Delta_{i+1}(\hf) \leq \ell^{j-i-1}\Delta_j(\hf).
 \]
 Since $E\cap F_0\neq \emptyset$ for all $E\in \hf(A)$,
 \[
 |\hf(A)|\leq \sum_{1\leq t\leq \ell} |\hf(A\cup \{x_t\})|\leq \ell^{j-i} \Delta_j(\hf).
 \]
As $A$ was chosen arbitrarily, we conclude that \eqref{ineq-fa} holds.
\end{proof}

\begin{thm}\label{thm-main2}
For $\ell\geq 4$,
\[
m_{\ell-1}(\ell)\leq 2\ell^{\ell-3}.
\]
\end{thm}

\begin{proof}
Assume that $\hf\subset \binom{[n]}{\ell}$ is an intersecting family with $\tau(\hf)\geq \ell$. Set $r=\Delta_{\ell-1}(\hf)$. We distinguish two cases:

{\bf Case 1.} $r\geq 2$.

Assume that $Y\cup \{x_i\}\in \hf$ for $i=1,\ldots,r$, $|Y|=\ell-1$. Note that  $r\leq \ell-1$. Indeed, otherwise if $r \geq \ell+1$ then $|\hf(\overline{Y})|=0$, contradicting $\tau(\hf)\geq \ell$. If  $r = \ell$ then $|\hf(\overline{Y})|=1<2\ell^{\ell-3}$  and we are done.
Then by applying \eqref{ineq-fa} with $j=\ell-1$
\begin{align*}
\gamma_{\ell-1}(\hf)\leq |\hf(\overline{Y})|= |\hf(\{x_1,\ldots,x_r\})|\leq  \Delta_{r}(\hf)\leq \ell^{\ell-1-r}\Delta_{\ell-1}(\hf)= r\ell^{\ell-1-r}.
\end{align*}
As $f(r):=r\ell^{\ell-1-r}$ is decreasing and $r\geq 2$, $\gamma_{\ell-1}(\hf)\leq 2\ell^{\ell-3}$.

{\bf Case 2.} $r= 1$.

Let $d=\Delta_{\ell-2}(\hf)$ and assume $|\hf(D)|=d$ for some $D$, $|D|=\ell-2$. Note that $r=1$ implies that $\hf(D)$ consists of $d$ pairwise disjoint 2-sets.  Hence $|\hf(\overline{D})|\leq 2^d\Delta_{d}(\hf)$. Since $\tau(\hf)\geq \ell$, we infer that $d\leq \ell$. If $d=\ell$, choose $x$ from one of the 2-sets in $\hf(D)$,  then for $\ell\geq 4$
\[
\gamma_{\ell-1}(\hf)\leq |\hf(\overline{D\cup\{x\}})|\leq 2^{\ell-1}\leq 2\ell^{\ell-3}.
 \]
 If $d=\ell-1$ then by $\Delta_{\ell-1}(\hf)=1$ we obtain that
 \[
 \gamma_{\ell-1}(\hf)\leq |\hf(\overline{D})|\leq 2^{\ell-1}\leq 2\ell^{\ell-3}.
 \]
 Thus we may assume that $d\leq \ell-2$. Using \eqref{ineq-fa} with $i=d$, $j=\ell-2$
\[
\gamma_{\ell-1}(\hf)\leq |\hf(\overline{D})|\leq 2^d\Delta_{d}(\hf) \leq 2^d\ell^{\ell-2-d}\Delta_{\ell-2}(\hf) =d2^d\ell^{\ell-2-d}.
\]
For $d\geq 2$ the RHS is at most $2\times 2^2\ell^{\ell-4}=8\ell^{\ell-4}\leq  2\ell^{\ell-3}$ for $\ell\geq 4$.

If $d=1$ then for any $E\cup \{x\}\in \hf$, by the intersecting property  and \eqref{ineq-fa}
\[
\gamma_{\ell-1}(\hf)\leq |\hf(\overline{E})|\leq |\hf(x)|\leq \Delta_1(\hf)\leq \ell^{\ell-2-1}\Delta_{\ell-2}(\hf) \leq \ell^{\ell-3}.
\]
\end{proof}

Let us mention that by \eqref{ineq-key5} the results of \cite{F19} and  \cite{Zakharov} imply much stronger bounds. However those are valid only for $\ell$ astronomically large.

Let $\ha\subset 2^{X}$ and $\hb\subset 2^{[n]}$. Let us define the {\it wreath product} $\ha\circ \hb$. Let $Z=X_1\cup X_2\cup\ldots\cup X_n$, $|Z|=n|X|$, $X_i$ is a copy of $X$ and let $\ha_i\subset 2^{X_i}$ be an isomorphic copy of $\ha$. For every edge $B=(i_1,\ldots,i_r)\in \hb$ define
\[
\ha^{B}:=\{A_{i_1}\cup \ldots\cup A_{i_r}\colon A_{i_1}\in \ha_{i_1},\ldots,A_{i_r}\in \ha_{i_r}\}.
\]
Note that $|\ha^{B}|=|\ha|^{|B|}$. Then define $\ha\circ\hb=\mathop{\cup}\limits_{B\in \hb} \ha^{B}$.

If $\ha$ is $k$-uniform and $\hb$ is $\ell$-uniform then $\ha\circ\hb$ is $k\ell$-uniform with $|\ha\circ\hb|=|\ha|^\ell|\hb|$.

\begin{claim}\label{claim-1}
If $\ha$ and $\hb$ are intersecting, then $\ha\circ\hb$ is intersecting and
\[
\tau(\ha\circ\hb) = \tau(\ha)\tau(\hb).
\]
\end{claim}

\begin{proof}
Suppose that $T$ is a transversal of $\ha\circ\hb$ and assume indirectly that $|T|<\tau(\ha)\tau(\hb)$.
 Define $t_i=|T\cap X_i|$ and $P=\{i\colon t_i\geq \tau(\ha)\}$. By the indirect assumption $|P|<\tau(\hb)$. Hence $\exists B\in \hb$ with $B\cap P= \emptyset$. I.e., for each $i\in B$, $t_i<\tau(\ha_i)=\tau(\ha)$. Consequently, for each $i\in B$ we can fix $A_i\in \ha_i$ with $A_i\cap T=A_i\cap (X_i\cap T)=\emptyset$. Thus $(\mathop{\cup}\limits_{i\in B} A_i)\cap T=\emptyset$, a contradiction. Thus $\tau(\ha\circ\hb) \geq  \tau(\ha)\tau(\hb)$.

 To prove $\tau(\ha\circ\hb)\leq  \tau(\ha)\tau(\hb)$ fix a transversal $(j_1,\ldots,j_s)$ of $\hb$ with $s=\tau(\hb)$.  For each $1\leq i\leq s$ fix a transversal $T_i$ of $\ha_{j_i}$ with $|T_i|=\tau(\ha)$. Now it is easy to verify that $\mathop{\cup}\limits_{1\leq i\leq s} T_i$ is a transversal of $\ha\circ\hb$.
\end{proof}

\begin{lem}
For $k\leq p\leq q$,
\begin{align}\label{ineq-key3}
\binom{p}{k}\binom{q}{k} > \binom{p-1}{k}\binom{q+1}{k}.
\end{align}
\end{lem}

\begin{proof}
Note that \eqref{ineq-key3}  is equivalently to
\[
\frac{\binom{p}{k}}{\binom{p-1}{k}}=\frac{p}{p-k}> \frac{q+1}{q+1-k}=\frac{\binom{q+1}{k}}{\binom{q}{k}}.
\]
The inequality  in the middle is $1+\frac{k}{p-k}> 1+\frac{k}{q+1-k}$, which is true by $p\leq  q$.
\end{proof}

\begin{lem}
For $k\geq 24$,
\begin{align}\label{ineq-key4}
\binom{\frac{4k+1}{3}}{k}\binom{\frac{4k-5}{3}}{k}> \binom{2k-1}{k}.
\end{align}
\end{lem}

\begin{proof}
We prove \eqref{ineq-key4} by induction on $k$. It can be checked by Mathematica that \eqref{ineq-key4} holds for $k=24,25,26$. Now we assume that \eqref{ineq-key4} holds for $k\geq 24$ and prove it for $k+3$. Note that
\begin{align*}
\frac{\binom{\frac{4k+1}{3}+4}{k}\binom{\frac{4k-5}{3}+4}{k}}{\binom{\frac{4k+1}{3}}{k}\binom{\frac{4k-5}{3}}{k}} &=\frac{\frac{4k+13}{3}\frac{4k+10}{3}\frac{4k+7}{3}\frac{4k+4}{3}}{\frac{k+13}{3}\frac{k+10}{3}\frac{k+7}{3}\frac{k+4}{3}}
\cdot \frac{\frac{4k+7}{3}\frac{4k+4}{3}\frac{4k+1}{3}\frac{4k-2}{3}}{\frac{k+7}{3}\frac{k+4}{3}\frac{k+1}{3}\frac{k-2}{3}}\\[5pt]
&=\frac{(4k+13)(4k+10)(4k+7)^2(4k+4)^2(4k+1)(4k-2)}{(k+13)(k+10)(k+7)^2(k+4)^2(k+1)(k-2)}\\[5pt]
&> \left(\frac{4k+13}{k+13}\right)^8> 2^8
\end{align*}
and
\begin{align*}
\frac{ \binom{2k+5}{k+3}}{ \binom{2k-1}{k}} =\frac{(2k+5)(2k+4)(2k+3)(2k+2)(2k+1)2k}{(k+3)(k+2)^2(k+1)^2k}
&=\frac{8(2k+5)(2k+3)(2k+1)}{(k+3)(k+2)(k+1)}\leq 2^{6}.
\end{align*}
It follows that
\[
\frac{\binom{\frac{4k+1}{3}+4}{k}\binom{\frac{4k-5}{3}+4}{k}}{\binom{\frac{4k+1}{3}}{k}\binom{\frac{4k-5}{3}}{k}}
>\frac{ \binom{2k+5}{k+3}}{ \binom{2k-1}{k}}.
\]
By the induction hypothesis,
\[
\frac{\binom{\frac{4k+1}{3}+4}{k}\binom{\frac{4k-5}{3}+4}{k}}{\binom{2k+5}{k+3}}
>\frac{ \binom{\frac{4k+1}{3}}{k}\binom{\frac{4k-5}{3}}{k}}{ \binom{2k-1}{k}}>1
\]
and \eqref{ineq-key4} is proven.
\end{proof}

\begin{prop}
For $k\geq 6$,
\begin{align}\label{ineq-3-3}
m_{2k-1}(2k)\geq 2\binom{2k-2}{k}+1.
\end{align}
\end{prop}

\begin{proof}
Let $X_1$, $X_2$, $X_3$ be disjoint copies of $[2k-1]$ and $X:=X_1\cup X_2\cup X_3$ the underlying set of $\ha\circ\hb$ where $\ha=\binom{[2k-1]}{k}$, $\hb=\binom{[3]}{2}$. Clearly, $\ha\circ\hb$ is $2k$-uniform, intersecting and by Claim \ref{claim-1}, $\tau(\ha\circ\hb)=2k$. To prove \eqref{ineq-3-3} we show
\begin{align}\label{ineq-triangle}
\gamma_{2k-1}(\ha\circ\hb)=2\binom{2k-2}{k}+1 \mbox{ for } k\geq 6.
\end{align}

Let $P\in \binom{X}{2k-1}$ be a set for which $|\ha\circ\hb(\overline{P})|$ is minimal. Define $a=|X_1\setminus P|$, $b=|X_2\setminus P|$, $c=|X_3\setminus P|$ and note $a+b+c=|X|-|P|=4k-2$. By symmetry assume $a\leq b\leq c$. Note also that
\begin{align}\label{ineq-polygon}
 |\ha\circ\hb(\overline{P})| =&\binom{a}{k}\binom{b}{k}+\binom{a}{k}\binom{c}{k}+\binom{b}{k}\binom{c}{k}:= f(a,b,c).
\end{align}

 For $k=6,7,\ldots,23$, by using Mathematica one can check that the minimum of $f(a,b,c)$ is $2\binom{2k-2}{k}+1$ under conditions $a+b+c=4k-2$ and $0\leq a\leq b\leq c\leq 2k-1$. Thus we assume $k\geq 24$ in the rest of the proof.

If $a\leq  k-1$ then the RHS of \eqref{ineq-polygon} reduces to $\binom{b}{k}\binom{c}{k}$. Hence the monotonicity of $\binom{m}{k}$ as a function of $m$ and the minimal choice of $P$ imply $a\geq k-1$.

Suppose first $a=k-1$. The RHS of \eqref{ineq-polygon} reduces to $\binom{b}{k}\binom{c}{k}$, where $k\leq b\leq c$, $b+c=3k-1$. For $b=k$ its value is $\binom{2k-1}{k}\geq 2\binom{2k-2}{k}+1$. For  $k+1\leq b\leq c$, by \eqref{ineq-key3} we have
\[
\binom{b}{k}\binom{c}{k}\geq \binom{k}{k}\binom{2k-1}{k}\geq 2\binom{2k-2}{k}+1.
\]

Now assume that $k\leq a\leq b\leq c$ and distinguish two cases.

{\bf Case 1.} $a=b$.

Set $a=b=x$. Then $c=4k-2-2x$ and
\[
 f(a,b,c) = 2\binom{x}{k}\binom{4k-2-2x}{k}+\binom{x}{k}\binom{x}{k}\geq 2\binom{x}{k}\binom{4k-2-2x}{k}+1.
\]
Let $g(x)=\binom{x}{k}\binom{4k-2-2x}{k}$. Note that
\begin{align*}
\frac{g(x+1)}{g(x)} =\frac{\binom{x+1}{k}\binom{4k-4-2x}{k}}{\binom{x}{k}\binom{4k-2-2x}{k}}
&=\frac{(x+1)(3k-2-2x)(3k-3-2x)}{(x+1-k)(4k-2-2x)(4k-3-2x)}
\end{align*}
and
\begin{align*}
&\quad (x+1)(3k-2-2x)(3k-3-2x)-(x+1-k)(4k-2-2x)(4k-3-2x) \\[5pt]
&= k (8x^2+(19-23k)x+16k^2-27k+11):=k \cdot h(x).
\end{align*}
Since $h(x)$ is a quadratic function  with axis of symmetry $x=\frac{23k-19}{16}$ and $k\leq x\leq \frac{4k-2}{3}\leq\frac{23k-19}{16}$ for $k\geq 5$, $h(x)$ is decreasing on $[k, \frac{4k-2}{3}]$. Moreover, for $k\geq 7$
\[
h(k)=k(k^2-8k+11)>0 \mbox{ and }  h\left(\frac{4k-2}{3}\right)=-\frac{1}{9} k (4k^2+5k-17)<0.
\]
We infer that $g(x)$ is a concave function and
\[
g(x) \geq \min\left\{g(k),g\left(\frac{4k-2}{3}\right)\right\}= \min\left\{\binom{2k-2}{k},\binom{\frac{4k-2}{3}}{k}^2\right\}.
\]
Since  by \eqref{ineq-key3} and \eqref{ineq-key4}
\[
\binom{\frac{4k-2}{3}}{k}^2 \geq \binom{\frac{4k+1}{3}}{k}\binom{\frac{4k-5}{3}}{k}> \binom{2k-1}{k}>\binom{2k-2}{k},
\]
we conclude that
\[
f(a,b,c)\geq 2g(x)+1 \geq 2\binom{2k-2}{k}+1.
\]

{\bf Case 2.} $a<b$.

By $k\leq a<b\leq c$ and $a+b+c=4k-2$,
\[
 f(a,b,c) \geq \binom{b}{k}\binom{c}{k} \overset{\eqref{ineq-key3}}{\geq} \binom{b-(b-a-1)}{k}\binom{c+(b-a-1)}{k}=\binom{a+1}{k}\binom{4k-3-2a}{k}.
\]
Let $x=a+1$ and  $g(x) = \binom{x}{k}\binom{4k-1-2x}{k}$.
Note that
\begin{align*}
\frac{g(x+1)}{g(x)} =\frac{\binom{x+1}{k}\binom{4k-3-2x}{k}}{\binom{x}{k}\binom{4k-1-2x}{k}}
&=\frac{(x+1)(3k-1-2x)(3k-2-2x)}{(x+1-k)(4k-1-2x)(4k-2-2x)}
\end{align*}
and
\begin{align*}
&\quad (x+1)(3k-1-2x)(3k-2-2x)-(x+1-k)(4k-1-2x)(4k-2-2x) \\[5pt]
&= k (8x^2+(13-23k)x+16k^2-19k+5):=k \cdot h(x).
\end{align*}
Since $h(x)$ is a quadratic function with axis of symmetry $x=\frac{23k-13}{16}$ and $k< x\leq \frac{4k+1}{3}<\frac{23k-13}{16}$ for $k\geq 5$, $h(x)$ is decreasing on $[k, \frac{4k+1}{3}]$. Moreover, for $k\geq 6$
\[
h(k)=k^2-6k+5>0 \mbox{ and }  h\left(\frac{4k+1}{3}\right)=-\frac{4}{9}  (k^2+5k-23)<0.
\]
We infer that $g(x)$ is a concave function and
\[
g(x) \geq \min\left\{g(k),g\left(\frac{4k+1}{3}\right)\right\}= \min\left\{\binom{2k-1}{k},\binom{\frac{4k+1}{3}}{k}\binom{\frac{4k-5}{3}}{k}\right\}.
\]
By \eqref{ineq-key4} we conclude that
\[
f(a,b,c) \geq g(x) \geq \binom{2k-1}{k} \geq 2\binom{2k-2}{k}+1.
\]
\end{proof}

\begin{prop}
\begin{align}
m_5(6) \geq 20.
\end{align}
\end{prop}

\begin{proof}
Recall that
\[
\hht_0=\{(1,2,3), (1,2,4),(3,4,5),(3,4,6),(1,5,6),(2,5,6),(1,3,5),(2,4,5),(1,4,6),(2,3,6)\}.
\]
Let $\hf=\hht_0\circ\binom{[3]}{2}$ and let $X_1\uplus X_2\uplus X_3$, $|X_i|=6$, be the ground set. Clearly, $|\hf|=3\times 10^2=300$.

For any $Y\subset X_1\uplus X_2\uplus X_3$, $|Y|=5$, we claim that $|\hf(\overline{Y})|\geq 20$. Without loss of generality, assume $|Y\cap X_1|\geq |Y\cap X_2|\geq |Y\cap X_3|$ and let $(a,b,c)=(|Y\cap X_1|, |Y\cap X_2|, |Y\cap X_3|)$. There are three cases: $(3,2,0), (3,1,1), (2,2,1)$. For $(3,2,0)$, by $\gamma_2(\hht_0)=2$ $|\hf(\overline{Y})|\geq 2\times |\hht_0|=20$.
For $(3,1,1)$, by $\gamma_1(\hht_0)=5$ $|\hf(\overline{Y})|\geq 5\times 5=25$. For $(2,2,1)$, by $\gamma_2(\hht_0)=2$ and $\gamma_1(\hht_0)=5$, $|\hf(\overline{Y})|\geq 2\times 5+2\times 5+2\times 2=24$. Thus $|\hf(\overline{Y})|\geq 20$.
\end{proof}

\begin{prop}
\begin{align}
m_4(5) \geq 6.
\end{align}
\end{prop}

\begin{proof}
Let $\hp\subset \binom{\{0,1,2,3,4\}}{2}$ be the {\it pentagon} $\{(i,i+1)\colon i=0,1,2,3,4\}$. Let $\hr=\{\{i,i+1,i+3\}\colon 0\leq i\leq 4\}$ with computation modulo 5. Note that $\hp$ and $\hr$ are cross-intersecting.

Let $X_j$, $0\leq j\leq 2$ be disjoint copies of $[0,4]$ and $\hp_j$, $\hr_j$ the corresponding copies of $\hp$ and $\hr$. Define $\hh\subset \binom{X_0\cup X_1\cup X_2}{5}$ as $\hh=(\hp_0\times \hr_1)\cup (\hp_1\times \hr_2)\cup (\hp_2\times \hr_0)$ where, e.g., $\hp_0\times \hr_1=\left\{P_0\cup R_1\colon P_0\in \hp_0,R_1\in \hr_1\right\}$. Note that $\hh$ is intersecting, $|\hh|=3\times 5^2=75$.

Consider an arbitrary 4-set $Y\subset X_0\cup X_1\cup Y_2$.

{\bf Case 1.} $|Y\cap X_i|\leq 1$ for two values of $i$.

Without loss of generality, assume that $|Y\cap X_1|\leq 1$, $|Y\cap X_2|\leq 1$. Then there are (at least) three edges $P,P',P''\in \hp_1$ disjoint to $Y$ and (at least) two edges $R,R'\in \hr_2$ disjoint to $Y$. Hence $|\hh(\overline{Y})|\geq 3\times 2=6$.

{\bf Case 2.} $|Y\cap X_i|= 2$ for two values of $i$.

Without loss of generality, assume that $|Y\cap X_0|= 2$, $|Y\cap X_1|=2$, $Y\cap X_2=\emptyset$. If $Y\cap X_1\in \hp_1$ then there are two edges $P_1,P_1'\in \hp_1$ disjoint to $Y$. It follows that $P_1\cup R_2$, $P_1'\cup R_2\in \hh(\overline{Y})$ for all $R_2\in \hr_2$. Thus
$|\hh(\overline{Y})|\geq |(\hp_1\times \hr_2)(\overline{Y})|\geq 2\times 5 =10$ and  we are done.

Now assume that $Y\cap X_1\notin \hp_1$ and let $P_1$ be the only edge in $\hp_1$ disjoint to $Y$. Since  $P_1\cup R_2\in \hh(\overline{Y})$ for all $R_2\in \hr_2$, $|(\hp_1\times \hr_2)(\overline{Y})|\geq 5$. If $Y\cap X_0\notin \hp_0$ then we can find $R_0\in \hr_0$ with $Y\cap R_0=\emptyset$ and $P_2\cup R_0\in \hh(\overline{Y})$ for all $P_2\in \hp_2$.  Thus $|\hh(\overline{Y})|\geq |(\hp_1\times \hr_2)(\overline{Y})|+|(\hp_2\times \hr_0)(\overline{Y})|\geq 5+5=10$.

If $Y\cap X_0\in \hp_0$ then there are $P_0,P_0'\in \hp_0$ disjoint to $Y$. Recall that $Y\cap X_1\notin \hp_1$. Then $R_1:=X_1\setminus Y\in \hr_1$. It follows that $P_0\cup R_1$, $P_0'\cup R_1\in (\hp_0\times \hr_1)(\overline{Y})$. Thus $|\hh(\overline{Y})|\geq |(\hp_1\times \hr_2)(\overline{Y})|+|(\hp_0\times \hr_1)(\overline{Y})|\geq 5+2=7$.
\end{proof}

\section{Concluding remarks}

Half a century ago Erd\H{o}s and Lov\'{a}sz \cite{EL} introduced the problem of investigating $m_0(k)$, the maximum number of edges in an intersecting $k$-graph with covering number $k$. This is a very different problem and unfortunately the exact value of $m_0(k)$ is not known even for $k=4$.

In Definition \ref{defn-1} we introduced the corresponding $\ell$-diversity, $m_{\ell}(k)$. Let us  make a rather audacious conjecture:

\begin{conj}
For all $\ell\geq 0$,
\begin{align}\label{conj-6.1}
\lim_{k\rightarrow \infty} \frac{m_{\ell+1}(k)}{m_{\ell}(k)}=0.
\end{align}
\end{conj}

Actually the first proofs (cf. \cite{Cherkashin}, \cite{AT}) of $m_0(k)=o(k^k)$ are implicitly showing $m_1(k)=o(k^k)$. This easily implies
\[
m_0(k)\leq k^{k-1}+m_1(k)=o(k^k).
\]

The currently known largest examples seem to support \eqref{conj-6.1} in the case $\ell=0$. For the case $\ell\geq 1$ very little is known.

Our main results connect $m_{\ell}(\ell+1)$ and the maximum $\ell$-diversity $g_{\ell}(n,k)$ of saturated intersecting $k$-graphs on $n$ vertices. Theorem \ref{main-1} establishes the exact value
\[
g_2(n,k)=2\binom{n-5}{k-3}-\binom{n-7}{k-5}   \mbox{ for } n\geq 13k^2.
\]
With a more detailed analysis of the branching process we could improve the constant 13 to 9. However the real challenge would be to prove the same result for $n\geq ck$ with a relatively small absolute constant $c$. No doubt to achieve that one would need some different methods.

In Proposition \ref{prop-2-2} we proved $m_3(4)=3$ thereby establishing $g_3(n,k)=3\binom{n-7}{k-4}+O(n^{k-5})$. Let us close this paper with the following problem.

\begin{prob}
For all $k\geq 5$ and $n>n_0(k)$ determine the exact value of $g_3(n,k)$.
\end{prob}

We believe that the optimal hypergraph contains $\hf_{\hl_3}$. However, unlike the Fano plane, $\hl_3$ is not 3-chromatic. As a matter of fact for $k\geq 6$ one can add to $\hf_{\hl_3}$ quite a number of 6-element sets along with their $k$-element supersets maintaining the intersecting property and increasing the triple-diveristy.

\end{document}